\documentclass[a4paper,10pt]{amsart}
\usepackage[utf8]{inputenc}
\usepackage[T1]{fontenc}
\usepackage[english]{babel}
\usepackage{amsmath,amssymb,amsthm}
\usepackage{bbm}
\usepackage{calrsfs}

\usepackage[pdftex]{color}
\usepackage[dvipsnames]{xcolor}
\usepackage[bookmarks=true,hyperindex,pdftex,colorlinks, citecolor=MidnightBlue,linkcolor=MidnightBlue, urlcolor=MidnightBlue
]{hyperref}

\usepackage{tikz}
\usetikzlibrary{matrix}
\usetikzlibrary{arrows}
\tikzset{>=latex}

\usepackage{ulem} 

\numberwithin{equation}{section}

\theoremstyle{plain}
\newtheorem{theorem}{Theorem}[section]
\newtheorem{corollary}[theorem]{Corollary}
\newtheorem{cor}[theorem]{Corollary}

\newtheorem{proposition}[theorem]{Proposition}
\newtheorem{lemma}[theorem]{Lemma}

\newtheorem{result}[theorem]{Result}

\theoremstyle{definition}

\newtheorem{remark}[theorem]{Remark}
\newtheorem{example}[theorem]{Example}
\newtheorem{examples}[theorem]{Examples}
\newtheorem{definition}[theorem]{Definition}

\newtheorem{problem}[theorem]{Problem}
\newtheorem{observation}[theorem]{Observation}

 \DeclareMathOperator{\dist}{dist}
 \DeclareMathOperator{\supp}{supp}
 
 \DeclareMathOperator{\Lin}{span}
 \DeclareMathOperator{\cone}{cone}
 \DeclareMathOperator{\co}{conv}
 \DeclareMathOperator{\cco}{\overline{conv}}

\newcommand{\R}{\mathbb{R}}
\newcommand{\N}{\mathbb{N}}

\newcommand{\D}{\mathbb{D}}

\newcommand{\C}{\mathbb{C}}

\newcommand{\mycolon}{{:}\allowbreak\ }

\newcommand{\eps}{\varepsilon}

\renewcommand{\leq}{\leqslant}
\renewcommand{\geq}{\geqslant}
\renewcommand{\le}{\leqslant}
\renewcommand{\ge}{\geqslant}

\newcommand{\Abs}[1]{\ensuremath{\left|#1\right|}}

\DeclareMathOperator{\NA}{NA}

\begin{document}

\title{Norm attaining operators of finite rank}

\author[V.~Kadets]{Vladimir Kadets}
\author[G.~L\'opez]{Gin\'es L\'opez}
\author[M.~Mart\'{\i}n]{Miguel Mart\'{\i}n}
\author[D.~Werner]{Dirk Werner}

\address[Kadets]{School of Mathematics and Informatics \\
  V.~N.~Karazin   Kharkiv National University \\
pl.~Svobody~4 \\
61022~Kharkiv \\ Ukraine
\newline
\href{http://orcid.org/0000-0002-5606-2679}{ORCID: \texttt{0000-0002-5606-2679} }
}
\email{vova1kadets@yahoo.com}

\address[L\'opez]{Departamento de An\'{a}lisis Matem\'{a}tico \\ Facultad de
 Ciencias \\ Universidad de Granada \\ 18071 Granada, Spain
\newline
\href{http://orcid.org/0000-0002-3689-1365}{ORCID: \texttt{0000-0002-3689-1365} }
 }
\email{glopezp@ugr.es}

\address[Mart\'{\i}n]{Departamento de An\'{a}lisis Matem\'{a}tico \\ Facultad de
 Ciencias \\ Universidad de Granada \\ 18071 Granada, Spain
\newline
\href{http://orcid.org/0000-0003-4502-798X}{ORCID: \texttt{0000-0003-4502-798X} }
 }
\email{mmartins@ugr.es}

\address[Werner]{Department of Mathematics \\ Freie Universit\"at Berlin \\
Arnimallee~6 \\ D-14195~Berlin \\ Germany\newline
\href{http://orcid.org/0000-0003-0386-9652}{ORCID: \texttt{0000-0003-0386-9652}}
}
\email{werner@math.fu-berlin.de}



\thanks{The research of the first author was done in the framework of the Ukrainian Ministry of Science and Education Research Program 0118U002036. The research of the second and third authors was partially supported by projects MTM2015-65020-P (MINECO/FEDER, UE) and PGC2018-093794-B-I00 (MCIU/AEI/FEDER, UE) and by the grant FQM-185 (Junta de Andaluc\'{\i}a/FEDER, UE). Part of the research was done during several stays of the first and fourth authors to the University of Granada which have been supported by the project MTM2015-65020-P (MINECO/FEDER, UE)}

\begin{abstract}
We provide sufficient conditions on a Banach space $X$ in order that there exist norm attaining operators of rank at least two from $X$ into any Banach space of dimension at least two. For example, a rather weak such condition is the existence of a non-trivial cone consisting of norm attaining functionals on~$X$. We go on to discuss density of norm attaining operators of finite rank among all operators of finite rank, which holds for instance when there is a dense linear subspace consisting of norm attaining functionals on~$X$. In particular, we consider the case of Hilbert space valued operators where we obtain a complete characterization of these properties. In the final section we offer a candidate for a counterexample to the complex Bishop-Phelps theorem on~$c_0$, the first such counterexample on a certain complex Banach space being due to V.~Lomonosov.
\end{abstract}

\date{May 16th, 2019; revised September 23rd, 2019}

\dedicatory{Dedicated to the memory of Victor Lomonosov}

\maketitle
\thispagestyle{empty}


\section{Introduction}

Shortly after Bishop and Phelps's papers (\cite{BP1}, \cite{BP2})
on the density of norm attaining functionals on a Banach space had appeared, Lindenstrauss, in his seminal work \cite{Lin}, launched the study of norm attaining operators. Let us recall that a bounded linear operator $T\mycolon X\longrightarrow Y$ between Banach spaces, $T\in \mathcal{L}(X,Y)$, is called \textit{norm attaining} if there is some $x_0\in X$ with $\|x_0\|=1$ and $\|Tx_0\|=\|T\|$; in this case we write $T\in \NA(X,Y)$. Lindenstrauss introduced the following properties~(A) and~(B) of a Banach space: $X$ has~(A) if $\NA(X,Z)$ is dense in $\mathcal{L}(X,Z)$ for all~$Z$; and $Y$ has~(B) if $\NA(W,Y)$ is dense in $\mathcal{L}(W,Y)$ for all~$W$. Among many other results, he showed that reflexive spaces have~(A) as does $\ell_1$, and that $c_0$, $\ell_\infty$ and finite dimensional polyhedral spaces are examples of Banach spaces with property~(B) and, finally, that there are Banach spaces $(X,Y)$ such that $\NA(X,Y)$ is not dense in $\mathcal{L}(X,Y)$. Major progress was made by Bourgain \cite{Bou} who proved that spaces with the RNP have property~(A) and provided a certain converse result. The problem of whether Hilbert spaces have property~(B) was left open in \cite{Lin}, and it was solved only 25~years later by Gowers \cite[Appendix]{Gow}, who showed that none of the spaces $\ell_p$ for $1<p<\infty$ have~(B). This result was pushed out by M.~Acosta \cite{Aco-Edinburgh, Aco-contemporary} showing that neither infinite-dimensional $L_1(\mu)$ spaces nor any strictly convex infinite dimensional Banach space have property~(B). Finally, let us comment that even though there are many Banach spaces $X$ for which all compact linear operators from them can be approximated by norm attaining (finite rank) operators \cite{JoWo} (including $X=C_0(L)$ and $X=L_1(\mu)$), it was proved in \cite{M-JFA2014} that there exists a compact operator between certain Banach spaces which cannot be approximated by norm attaining operators. For more information and background on the topic of norm attaining operators we refer to the survey papers \cite{Aco} and \cite{M-RACSAM}.

One should observe that none of the negative results summed up above says anything about operators of finite rank. Actually, it is one of the major open questions in the theory whether all finite dimensional Banach spaces have Lindenstrauss's property~(B); equivalently, whether every finite rank operator between Banach spaces can be approximated by (finite rank) norm attaining operators. The aim of our paper is to contribute to this problem, in particular for rank-two operators.

In the case of linear functionals, it is clear from the Hahn-Banach theorem that $\NA(X)$, the set of norm attaining functionals, is always non-empty. For operators of rank two it is not clear at all whether there are norm attaining ones. We are going to investigate this problem in detail, both in general and in the particular case when the range space is the two-dimensional Hilbert space~$\ell_2^2$.

Let us outline the contents of the paper. We devote Section~\ref{sect:4conditions} to the study of the existence of norm attaining operators of finite rank.
For certain Banach spaces we prove that there are norm attaining operators of finite rank to all range spaces using
known sufficient conditions. A first new result says that whenever $\NA(X)$ contains $n$-dimensional subspaces for a Banach space $X$, there are rank~$n$ norm attaining operators from $X$ into any Banach space $Y$ with dimension at least~$n$. There are Banach spaces $X$ for which $\NA(X)$ contains no two-dimensional subspaces (this is proved in \cite{Rmoutil}, \cite{KLM} or \cite{KLMW} building on the ingenious ideas of Ch.~Read \cite{Read}), so this does not solve the existence problem for all domain spaces. However, we also show that it is sufficient for the existence of norm attaining rank two operators that $\NA(X)$ contains a nontrivial cone. We do not know whether this condition is also necessary; neither do we know whether every Banach space shares this property.

Section \ref{section:density} contains a new result on the density of norm attaining finite rank operators from a Banach space $X$: this is the case if $\NA(X)$ contains ``sufficiently many'' linear subspaces. In particular, this holds if $\NA(X)$ contains a dense linear subspace (for instance if $\NA(X)$ is actually a linear space itself). We use this result to recover  the known results from \cite{JoWo} and \cite{M-RACSAM} about the density  of norm attaining compact operators on a Banach space whose dual satisfies an appropriate version of the approximation property, like for example $C_0(L)$ spaces, $L_1(\mu)$ spaces, preduals of $\ell_1$, among many others. But it can also be  used to get some new results. Among other examples, we show that all finite rank operators from $X$ can be approximated by norm attaining operators in the following cases: $X$ is a finite-codimensional proximinal subspace of $c_0$ or of $\mathcal{K}(\ell_2)$, $X$ is a $c_0$-sum of reflexive spaces.

The special case when the range space is a two-dimensional Hilbert space is studied in Section~\ref{sect:naoperators}. Here, we characterize the norm attaining rank-two operators in $\mathcal{L}(X,\ell_2^2)$ in terms of the geometry of the dual norm on the set~$\NA(X)$. As a consequence, we show that the set of norm attaining rank-two operators in $\mathcal{L}(X,\ell_2^2)$ is not empty if and only if there are $f\in \NA(X)$ and $g\in X^*$ with $\|f\|=1$ and $0<\|g\|\leq 1$ such that $\Vert f+tg\Vert\leq \sqrt{1+t^2}$ for all $t\in\R$ and if and only if there is $f\in \NA(X)$ with $\|f\|=1$ such that the operator $f\otimes (1,0)\in \mathcal{L}(X,\ell_2^2)$ is not an extreme point in the unit ball of $\mathcal{L}(X,\ell_2^2)$. We do not know if such a norm attaining functional $f$ can be found on every Banach space $X$.

The last part of the paper, Section~\ref{sect:Lomonosov}, is devoted to
commenting on V.~Lomonosov's solution of the complex Bishop-Phelps problem, which is explained in the first few paragraphs of that section. We present a subset of the complex space $c_0$ which might be a candidate for a bounded closed convex subset without modulus attaining (complex) functionals, that is, a possible Lomonosov type example in~$c_0$.

\bigskip

We finish this introduction with the needed notation. We have already explained the notation $\mathcal{L}(X,Y)$, $\NA(X,Y)$, and $\NA(X)$. In addition, we define $\NA_1(X) := \{f\in \NA(X)\mycolon \|f\|=1\}$. For $k\in \N$ with $k\geq 2$, we also use the notation $\mathcal{L}^{(k)}(X,Y)$ (resp.\ $\NA^{(k)}(X,Y)$) for the subset of $\mathcal{L}(X,Y)$ (resp.\ $\NA(X,Y)$) consisting of operators of rank~$k$. As usual, $B_X=\{x\in X\mycolon \|x\|\le 1\}$ stands for the closed unit ball of~$X$, $S_X=\{x\in X\mycolon \|x\|=1\}$ for its unit sphere and, less canonically,
$U_X=\{x\in X\mycolon \|x\|<1\}$ for its open unit ball. Further needed notation is the following: $M^{\bot}$ is the annihilator in $X^*$ of a closed subspace $M$ of $X$, $J_X\mycolon X\longrightarrow X^{**}$ denotes the canonical isometric inclusion of a Banach space into its bidual, $\mathcal{K}(X,Y)$ is the space of compact linear operators between $X$ and $Y$,
$\cone\{f,g\}$ stands for the cone generated by $f$ and~$g$, that is, $\cone\{f,g\}= \{ af+bg\mycolon a,b\ge0\}$.

If $x_0\in X$ is a non-zero vector, any functional $f\in S_{X^*}$ with $f(x_0)=\|x_0\|$ is called a \textit{supporting functional} at~$x_0$.
(Obviously the supporting functionals are precisely the norm attaining ones.)
If $x_0\in S_X$ admits a unique supporting functional, it is called a \textit{smooth point}. A theorem due to Mazur guarantees that for separable (in particular finite-dimensional) $X$ the set of smooth points is dense in~$S_X$ \cite[Th.~20F]{Holmes}.

Finally, let us remark that the spaces considered in this paper are Banach spaces over the reals, with the exception of Section~\ref{sect:Lomonosov} where complex Banach spaces are the issue.

\section{Existence of norm attaining finite rank operators}\label{sect:4conditions}
Before asking for the density of finite rank norm attaining operators, we should ask for the existence of such operators. It is not clear, to the best of our knowledge, that for all Banach spaces $X$ and $Y$ of dimension at least two, there exists a norm attaining operator from $X$ to $Y$ with finite rank greater than one. Our goal in this section is to discuss known and new sufficient conditions for the existence of norm attaining finite rank operators. In particular, we will focus  on the rank-two case. So the leading question here is the following.

\begin{problem}\label{problem:existence}
Is $\NA^{(2)}(X,Y)$ non-empty for all Banach spaces $X$ and $Y$ of dimension at least two?
\end{problem}

An obvious comment here is that for the above problem, it is enough to deal with range spaces $Y$ of dimension two. The next comment, though easy as well, is more surprising.

\begin{remark}\label{remark:toproblemexistence}
(a) {\slshape If $X$ is a Banach space and $\NA^{(2)}(X,\ell_2^2)$ is non-empty, then $\NA^{(2)}(X,Y)$ is also non-empty for every Banach space $Y$ of dimension at least two.}

  Indeed, we can assume that $Y$ is a two-dimensional Banach space. Now, take $T\in\NA^{(2)}(X,\ell_2^2)$ with $\|T\|=1$ and pick $x_0\in S_X$ such that $\Vert T(x_0)\Vert=1$. Fix an isomorphism $S$ from $\ell_2^2$ onto $Y$ with $\Vert S\Vert=1$. Then $S$ attains its norm, so there is $z\in S_{\ell_2^2}$ such that $\Vert S(z)\Vert=1$. Consider a rotation operator $\pi$ on $\ell_2^2$ such that $\pi(T(x_0))=z$. Then, $S\pi T\in\NA^{(2)}(X,Y)$ and so $\NA^{(2)}(X,Y)$ is non-empty.

(b) In fact, {\slshape if there exist $n\in \N$ and Banach spaces $X$ and $Y$ of dimension at least $n$ such that $\NA^{(n)}(X,Y)=\emptyset$, then
    $$
    \NA(X,\ell_2)\subset \bigcup\nolimits_{k\leq n-1}\mathcal{L}^{(k)}(X,\ell_2),
    $$
that is, every norm attaining operator from $X$ to $\ell_2$ has rank at most~${n-1}$.}

Indeed, if there exists $T\in \NA(X,\ell_2)$ whose rank is greater than or equal to $n$ (or even has infinite rank), then composing with a suitable orthogonal projection $P$ from $\ell_2$ to an $n$-dimensional subspace $H_n$ of $\ell_2$, we get that $PT\in \NA^{(n)}(X,H_n) \cong \NA^{(n)}(X,\ell_2^n)$. By an argument completely identical to the one given in item~(a), this provides that $\NA^{(n)}(X,Y)\neq \emptyset$.
\end{remark}

The above arguments explain the key role of $\ell_2^2$ to solve the problem of deciding whether $\NA^{(2)}(X,Y)$ is non-empty for all Banach spaces $X$ and $Y$ with $\dim(Y)\geq 2$. However, in the following we will study the problem for arbitrary range spaces~$Y$. In Section~\ref{sect:naoperators} we will  study  the particular case of $\NA^{(2)}(X,\ell_2^2)$ and we will even give characterizations of the statement that this set is non-empty.

But let us return to the general case of Problem~\ref{problem:existence}. First, we try to focus on the range space. If $Y$ is a polyhedral two-dimensional Banach space, then as a result of Lindenstrauss \cite[Proposition~3]{Lin}, $\NA(X,Y)$ is dense in $\mathcal{L}(X,Y)$ for all Banach spaces $X$, so the result is clear. Next, if $Y$ is not polyhedral and  not strictly convex either, then it is easy to construct a norm attaining rank-two operator from any Banach space $X$ of dimension greater than one into $Y$ (indeed, go first onto $\ell_\infty^2$ and then use that $S_Y$ contains a segment to produce an injective operator from $\ell_\infty^2$ into $Y$ that carries a whole maximal face of $S_{\ell_\infty^2}$ to the unit sphere of $Y$, see the proof of Proposition~\ref{prop:allNAimpliesybotNA-notrotund}). For strictly convex range spaces $Y$, we do not know the answer, even for $Y$ being a two-dimensional Hilbert space, and actually this case will be studied in depth in Section~\ref{sect:naoperators}, as announced before.

In this section we will mainly focus on the domain space. Our first comment is that, since compact operators are completely continuous, every compact operator whose domain is a reflexive space attains its norm (see \cite[p.~270]{M-RACSAM} for an argument). Next, there is an easy argument to get rank-two norm-attaining operators from a given Banach space $X$ having a one-complemented reflexive subspace of dimension greater than one to arbitrary range spaces~$Y$. Indeed, let $P\mycolon X\longrightarrow X$ be a norm-one projection such that $P(X)=Z$ is reflexive and $\dim(Z)\geq 2$. Now, every finite rank (actually compact) operator from a reflexive space is norm-attaining, so we just have to compose an arbitrary rank-two operator $S\mycolon Z \longrightarrow Y$ with the operator $P$ viewed as $P\mycolon X\longrightarrow P(X)=Z$ to get that $T=SP\in \mathcal{L}(X,Y)$ has rank-two and attains its norm (indeed,  $\|T\|=\|S\|$ and there exists $z\in S_Z$ such that $\|Sz\|=\|T\|$, but $z=P(z)\in S_X$ and so $\|Tz\|=\|SPz\|=\|Sz\|=\|T\|$). Actually, the same proof shows that every compact operator which factors through $P$ is norm attaining, but this is the same as requiring that the kernel of the operator contains the kernel of $P$.

\begin{result}[\mbox{\rm Folklore}]
Let $X$ be a Banach space, let $P\in \mathcal{L}(X,X)$ be a norm-one projection such that $P(X)$ is reflexive, and let $Y$ be an arbitrary Banach space. Then, every compact operator $T\mycolon X\longrightarrow Y$ for which $\ker P\subset \ker T$ attains its norm.
\end{result}

But to have one-complemented closed subspaces of dimension greater than one is a quite strong requirement, and there are even Banach spaces without  norm-one projections apart from the trivial ones (the identity and rank-one projections), see \cite{BoszGar} for a finite-dimensional example.

Anyway, a quick glance at the proof of the above result makes one realize that the only properties of the norm-one projection $P$ that we have used are that $P(X)$ is reflexive and that $P(B_X)=B_{P(X)}$, but not that $P^2=P$. As $P(X)=X/{\ker P}$, we may try to consider general quotient maps instead of projections. Let $X$ be a Banach space, let $Z$ be a closed subspace of $X$, and let $Y$ be an arbitrary Banach space. Suppose that $X/Z$ is reflexive (this  is usually referred to by saying that $Z$ is a \emph{factor reflexive} subspace of~$X$) and suppose also that the quotient map $q\mycolon X\longrightarrow X/Z$ satisfies that $q(B_X)=B_{X/Z}$, then every compact operator $T\mycolon X\longrightarrow Y$ such that $Z\subset \ker T$ attains its norm. Indeed, we may write $T=\widetilde{T}\circ q$ where $\|\widetilde{T}\|=\|T\|$ and $\widetilde{T}$ is compact, so there is $\xi \in B_{X/Z}$ such that $\|\widetilde{T}(\xi)\|=\|T\|$, but $\xi=q(x)$ for some $x\in B_X$ by hypothesis, so $\|T(x)\|=\|T\|$.

How to get the condition that $q(B_X)=B_{X/Z}$? This is just the proximinality of $Z$ in $X$. Recall that a (closed) subspace $M$ of $X$ is called \textit{proximinal} if for each $x\in X$ there is some $m\in M$ such that $\|x-m\|= \dist(x,M)$. We refer to the book \cite{Singer} for background. Clearly, $M$ is proximinal in $X$ if and only if $q(B_X)=B_{X/M}$, see \cite[Theorem~2.2]{Singer} for instance. Therefore, we have shown the following.

\begin{result}[\mbox{\rm Folklore}]\label{result:factor-reflexive-implies-existence}
Let $X$ be a Banach space, let $Z$ be a factor reflexive proximinal subspace of $X$, and let $Y$ be an arbitrary Banach space. Then, every compact operator $T\mycolon X\longrightarrow Y$ for which $Z\subset \ker T$ attains its norm.
\end{result}

The problem of whether every infinite-dimensional Banach space contains a two-codimensional proximinal subspace \cite[Problem~2.1]{Singer} was open until a celebrated example was  recently given by Read~\cite{Read}: there is a Banach space $\mathcal{R}$ containing no finite-codimensional proximinal subspaces of codimension greater than one (and then it contains no proximinal factor reflexive subspaces of infinite codimension either, use \cite[Proposition~2.3]{Narayana-Rao}). We refer the reader to \cite{KLM,KLMW,Rmoutil} for more information on Read type spaces. Therefore, Result~\ref{result:factor-reflexive-implies-existence} does not provide a complete positive solution of Problem~\ref{problem:existence}.

Our next step is to get a slightly weaker sufficient condition for norm attainment than the one given in Result~\ref{result:factor-reflexive-implies-existence}, which is new as far as we know. Namely, it is easy to see that if $Z$ is a factor reflexive proximinal subspace of a Banach space $X$, then $Z^\bot\subset \NA(X)$ (see \cite[Lemma~2.2]{Band-Godefroy} for instance), but the converse result is not true (see \cite[Section~2]{InduPLMS1982} or \cite[Section~2]{Rmoutil} for a discussion of this). Our result is that the condition $Z^\bot\subset \NA(X)$ is enough to get the conclusion of Result~\ref{result:factor-reflexive-implies-existence}.

\begin{proposition}\label{prop:kerbotinNAimpliesNA}
Let $X$ be a Banach space, let $Z$ be a closed subspace of $X$ such that $Z^\bot\subset \NA(X)$, and let $Y$ be an arbitrary Banach space. Then, every compact operator $T\in \mathcal{L}(X,Y)$ for which $Z\subset \ker T$ attains its norm.
\end{proposition}

\begin{proof}
As $Z^\bot\subset \NA(X)$, it is immediate from James's theorem that $X/Z$ is reflexive (see the proof of \cite[Lemma~2.2]{Band-Godefroy}). As $Z\subset \ker T$, the operator $T$ factors through $X/Z$, that is, there is an operator $\widetilde{T}\mycolon X/Z\longrightarrow Y$ such that $T=\widetilde{T}\circ q$, and it is clear that $\|\widetilde{T}\|=\|T\|$ and that $\widetilde{T}$ is compact whenever $T$ is. Then, $\widetilde{T}$ attains its norm (it is compact defined on a reflexive space), so also the adjoint $\widetilde{T}^*$ attains its norm. That is, there is $y^*\in S_{Y^*}$ such that $\|\widetilde{T}^*y^*\|=\|T\|$. Now, the functional $x^*=T^*y^*=\bigl[q^*\widetilde{T}^*\bigr](y^*)\in X^*$ vanishes on $Z$, so it belongs to $Z^\bot\subset \NA(X)$. This implies that there is $x\in S_X$ such that
$$
|x^*(x)|=\|x^*\|=\bigl\|\bigl[q^*\widetilde{T}^*\bigr](y^*)\bigr\| =\bigl\|q^*(\widetilde{T}^*y^*)\bigr\| = \bigl\|\widetilde{T}^*(y^*)\bigr\|=\|T\|,
$$
where we have used the immediate fact that $q^*$ is an isometric embedding as $q$ is a quotient map. Therefore, $\|T\|=|[T^*y^*](x)|=|y^*(Tx)|$ and so $\|Tx\|=\|T\|$, as desired.
\end{proof}

Observe that the proposition above can also be  written in the following more suggestive form.

\begin{corollary}
Let $X$, $Y$ be Banach spaces and let $T\in \mathcal{L}(X,Y)$ be a compact operator. If $\,[\ker T]^\bot\subset \NA(X)$, then $T$ attains its norm.
\end{corollary}

The following obvious consequence gives a solution to Problem~\ref{problem:existence} in most Banach spaces.

\begin{corollary}\label{corollary-NA-two-dimensional-subspaces}
Let $X$ be a Banach space. If $\NA(X)$ contains two-dimensional subspaces, then $\NA^{(2)}(X,Y)$ is non-empty for any Banach space $Y$ of dimension at least two.
\end{corollary}

What happens with Read's space $\mathcal{R}$? (Un)fortunately, Corollary~\ref{corollary-NA-two-dimensional-subspaces}   does not apply since $\NA(\mathcal{R})$ does not contain two-dimensional subspaces, as was shown by Rmoutil \cite[Theorem~4.2]{Rmoutil}. Actually, Rmoutil used the fact that if $Z$ is a finite-codimensional closed subspace of a Banach space $X$ such that $X/Z$ is strictly convex, then $Z^\bot\subset \NA(X)$ if and only if $Z$ is proximinal (see \cite[Lemma~3.1]{Rmoutil}). Then, he showed, for $X=\mathcal{R}$, that if $Z^\bot$ is contained in $\NA(\mathcal{R})$, then $\mathcal{R}/Z$ is strictly convex and so $Z$ is proximinal and hence it has codimension one. Actually, $\mathcal{R}^{**}$ is strictly convex \cite[Theorem~4]{KLM}, so all quotients of $\mathcal{R}$ are strictly convex.

What to do then with $\mathcal{R}$? Well, $\mathcal{R}$ is not smooth (this follows from the formula for the directional derivative of its norm given in \cite[Lemma~2.5]{Read}), so the following easy observation applies to it.

\begin{observation}\label{observation:smooth}
{\slshape If $X$ is a non-smooth Banach space, then $\NA^{(2)}(X,Y)$ is non-empty for any Banach space $Y$ with $\dim(Y)\geq 2$.}

Indeed, there are $x_0\in S_X$ and linearly independent $f,g\in S_{X^*}$ such that $f(x_0)=g(x_0)=1$. Consider two linearly independent vectors $y_1$ and $y_2$ of $S_Y$ and, replacing $y_2$ by $-y_2$ if necessary, let
$$
\alpha_0:=\|y_1+y_2\|=\max\{\|ay_1+by_2\|\mycolon |a|,|b|\leq 1\}.
$$
Now, define the operator $T\in \mathcal{L}(X,Y)$ by
$Tx=\frac{1}{\alpha_0}\bigl(f(x)y_1+g(x)y_2\bigr)$ for all $x\in X$ and observe that $T$ has rank two, that $\|T\|\leq 1$, and that $\|Tx_0\|=1$. Thus $T\in \NA^{(2)}(X,Y)$, giving the result.
\end{observation}

Observation~\ref{observation:smooth} solves Problem~\ref{problem:existence} for $\mathcal{R}$. But, are the already presented results applicable to solve the problem for all Banach spaces? The answer is no since we may construct a smooth renorming $\widetilde{\mathcal{R}}$ of $\mathcal{R}$ such that $\NA(\widetilde{\mathcal{R}})=\NA(\mathcal{R})$ \cite[Example~12]{KLM}, and neither Corollary~\ref{corollary-NA-two-dimensional-subspaces} nor Observation~\ref{observation:smooth} apply. Nevertheless, there is something these two results have in common: in both cases, the set of norm attaining functionals contains non-trivial cones. This also happens in $\NA(\widetilde{\mathcal{R}})$ (as it coincides with $\NA(\mathcal{R})$), and this will be the key to obtain the main new existence result about norm attaining rank-two operators.

\begin{theorem} \label{Theo2-implication}
Let $X$ be a real Banach space, let $f_1, f_2 \in S_{X^*}$ be linearly independent, $Z = \ker f_1 \cap \ker f_2$ and $\cone\{f_1, f_2\} \subset \NA(X)$. Then, for every real two-dimensional normed space $E$ there is a norm attaining surjective operator $T \mycolon X \longrightarrow E$ with $\ker T = Z$.
\end{theorem}

Before providing the proof of this result, let us give some consequences and comments.

Observe that Theorem~\ref{Theo2-implication} implies the following sufficient condition for the existence of norm attaining operators of rank two.

\begin{cor}\label{cor4.6}
Let $X$ be a Banach space. If there exist two linearly independent $f,g\in X^*$ such that $\cone\{f,g\}\subset \NA(X)$, then $\NA^{(2)}(X,Y)\neq \emptyset$ for every Banach space $Y$ of dimension at least two.
\end{cor}

We do not know whether the condition is necessary as well, and we don't know any Banach space that fails it.

\begin{problem}
Does $\NA(X)$ contain non-trivial cones for every infinite-dimensional Banach space $X$?
\end{problem}

\begin{problem}
Let $X$ be a Banach space and suppose that $\NA^{(2)}(X,Y)\neq \emptyset$ for every Banach space $Y$ of dimension at least two. Does this imply that $\NA(X)$ contains a non-trivial cone?
\end{problem}

As promised before we stated Theorem~\ref{Theo2-implication}, this result solves the problem of the existence of rank-two norm attaining operators for~$\widetilde{\mathcal{R}}$.

\begin{example}\label{example-R-widetilde_R}
{\slshape The Read space $\mathcal{R}$ given in \cite{Read} and its smooth renorming $\widetilde{\mathcal{R}}$ given in \cite[Example~12]{KLM}, satisfy that their set of norm attaining functionals contains non-trivial cones (but no non-trivial subspaces). Therefore, for every Banach space $Y$ of dimension at least two, both $\NA^{(2)}(\mathcal{R},Y)$ and $\NA^{(2)}(\widetilde{\mathcal{R}},Y)$ are non-empty.}

Indeed, as $\mathcal{R}$ is not smooth, taking linearly independent $f_1, f_2\in \mathcal{R}^*$ and $x_0\in S_{\mathcal{R}}$ such that $f_1(x_0)=1=f_2(x_0)$, it is immediate that $\cone\{f_1,f_2\}\subset \NA(\mathcal{R})$. For the space $\widetilde{\mathcal{R}}$, we just have to observe that $\NA(\widetilde{\mathcal{R}})=\NA(\mathcal{R})$, as shown in \cite[Example~12]{KLM}.
\end{example}

It is now time to present the proof of Theorem~\ref{Theo2-implication}. We first need some preliminary results. We recall that we denote the open unit ball of a Banach space $X$ by~$U_X$.

\begin{lemma} \label{prop1-arbitrary-point}
Let $E$ be a two-dimensional normed space, let $e_1 \in S_E$, $e_1^* \in S_{E^*}$ such that $e_1^*(e_1)=1$, and let $e_2 \in  S_E \cap \ker e_1^*$. For $0<\tau<1$, denote by $T_{\tau} \mycolon  E \longrightarrow E$ the norm-one linear operator such that $T_{\tau} (e_1) = e_1$ and $T_{\tau}(e_2) = \tau e_2$. Then, for every compact subset $K \subset E$ such that $\sup \Abs{e_1^*(K)} < 1$ there is $0<\tau <1$ such that $T_{\tau}(K) \subset U_E$.
\end{lemma}

\begin{proof}
Let $\tau_n = \frac{1}{n}$. Then the operators $T_{\tau_n}$ converge pointwise to the operator $T = e_1^* \otimes e_1$. Now, pointwise convergence on $E$ implies uniform convergence on $K$. Since, by hypothesis, $U_E$ is an open neighborhood of $T(K) = \{e_1^*(x)e_1 \mycolon x \in K\}$, there is $n \in \N$ such that  $T_{\tau_n}(K) \subset U_E$.
\end{proof}

\begin{lemma} \label{prop2-smooth-point-and-an-angle}
Under the conditions of the previous lemma, let additionally $e_1$ be a smooth point of $S_E$, and let $h_1, h_2 \in E^*$ be two linearly independent functionals such that $e_1^* = \frac12(h_1 + h_2)$ and $h_1(e_1) = h_2(e_1) = 1$. Denote
$$
A = \{x \in E \mycolon \max\{\Abs{h_1(x)}, \Abs{h_2(x)}\} \le 1\}.
$$
Then, there is $0<\tau<1$ such that $T_{\tau}(A) \subset B_E$.
\end{lemma}

\begin{proof}
Since $e_1$ is a smooth point of $B_E$, it follows by geometrical reasoning in the plane that there is a neighborhood $V$ of $e_1$ such that $A \cap V \subset B_E$ (i.e., the parallelogram $A$ touches $S_E$ at $e_1$ from the inside of $B_E$). Note that $e_1$ is a strongly exposed point of~$A$; it is strongly exposed by $e_1^*= \frac12(h_1+h_2)$.  Hence, there is some $\delta > 0$ such that
$$
A_\delta := \{x \in A \mycolon \Abs{e_1^*(x)}  > 1 - \delta\} \subset A\cap V\subset B_E.
$$
Let us apply Lemma~\ref{prop1-arbitrary-point} to $K = A \setminus A_\delta = \{x \in A \mycolon \Abs{e_1^*(x)}  \le 1 - \delta\}$. We obtain some  $0<\tau<1$ so that $T_{\tau}(K) \subset U_E$. Since $T_{\tau}(A_\delta) \subset A_\delta  \subset B_E$, this gives us the desired inclusion $T_{\tau}(A) \subset B_E$.
\end{proof}

\begin{lemma} \label{prop3-tangents-to-an-arc}
Let $Y$ be a two-dimensional normed space, $x_1, x_2 \in S_Y$ be linearly independent smooth points, and let the corresponding supporting functionals $x_1^*,  x_2^* \in S_{Y^*}$, $x_1^*(x_1) = x_2^*(x_2) = 1$, be also linearly independent. Let $y \in S_Y$ be of the form $y = a_1x_1 + a_2 x_2$ with $a_i > 0$, $i=1, 2$. Then, every supporting functional $f\in S_{Y^*}$ at $y$ belongs to $\cone\{x_1^*, x_2^*\}$.
\end{lemma}

\begin{proof}
We again argue  geometrically.
Denote $b_1 = x_1^*(x_2)$,  $b_2 = x_2^*(x_1)$. Evidently, $\Abs{b_i} < 1$, $i=1, 2$. Denote by $x_3$ the point at which $x_1^*(x_3) = x_2^*(x_3) = 1$ and consider the triangle $\Delta $ whose vertices are $x_1$, $x_2$, and $x_3$. Then, $y \in \Delta$, and the supporting line $\ell = \{x \in Y \mycolon f(x) = 1\}$ contains~$y$. Now, $x_1, x_2$ lie on one  side of $\ell$ (actually, all points of $B_Y$ do), whereas  $x_3$ lies on the opposite side of $\ell$, that is
\begin{equation*}
f(x_1) \le 1, \ f(x_2) \le 1, \  \textrm{and}\ f(x_3) \ge 1.
\end{equation*}
Since $x_1^*,  x_2^* \in Y^*$ are linearly independent, there is a (unique) representation of $f$ as $f = c_1 x_1^* + c_2 x_2^*$. Let us substitute this representation into the previous inequalities:
\begin{align} \label{eq: fandDelta-subs}
  c_1 + c_2b_2 \le 1, \\
  \label{eq: fandDelta-subs3}
  c_1 b_1 + c_2 \le 1, \\
  \label{eq: fandDelta-subs4}
c_1 + c_2   \ge 1.
\end{align}
From \eqref{eq: fandDelta-subs} and \eqref{eq: fandDelta-subs4} together with $b_2\neq1$, we deduce that $c_2 \ge 0$,
and likewise from \eqref{eq: fandDelta-subs3} and \eqref{eq: fandDelta-subs4}
  that  $c_1 \ge 0$.
\end{proof}

We are now able to present the pending proof.

\begin{proof}[Proof of Theorem~\ref{Theo2-implication}]
Denote by $q$ the corresponding quotient map $q \mycolon X \longrightarrow X/Z$; then $q^*$ maps $(X/Z)^*$ isometrically onto $Z^\bot = \Lin\{f_1, f_2\} \subset X^*$. Let $x_1, x_2 \in S_X$  be points at which $f_1(x_1) = f_2(x_2) = 1$, $\widetilde x_1 = q(x_1)$, $\widetilde x_2 = q(x_2)$ and let $\widetilde f_1, \widetilde f_2 \in S_{(X/Z)^*}$ be those functionals for which $q^*(\widetilde f_i) = f_i$, $i=1, 2$. Then, $\widetilde f_1(\widetilde x_1) = \widetilde f_2(\widetilde x_2) = 1$, so, in particular, $\widetilde x_1, \widetilde x_2 \in S_{X/Z}$.

We will consider two cases.

$\bullet$ {Case 1:} $\widetilde x_1$ and $\widetilde x_2$  are smooth points  of~$S_{X/Z}$. In this case, since $\widetilde f_1, \widetilde f_2$  are linearly independent, $\widetilde x_1, \widetilde x_2$  are linearly independent as well. Let $\ell \subset X/Z$ be the straight line connecting $\widetilde x_1$ with $\widetilde x_2$, let $\widetilde f_3 \in S_{(X/Z)^*}$ be the norm-one functional  taking a positive constant value $\alpha < 1$ on $\ell$ and let $\widetilde x_3 \in S_{X/Z}$ be a point at which $\widetilde f_3(\widetilde x_3) = 1$.   Select a point $e_1 \in S_E$,  a supporting functional $e_1^* \in S_{E^*}$ at $e_1$, $e_2 \in  S_E \cap \ker e_1^*$ and $T_{\tau} \mycolon  E \longrightarrow E$  as in Lemma~\ref{prop1-arbitrary-point}. Choose  $t \in (\alpha, 1)$, and denote by  $R \mycolon X/Z \longrightarrow E$ the linear operator such that $R(\widetilde x_1 - \widetilde x_2) = e_2$ and $ R(\widetilde x_1) = t e_1$.
Applying Lemma~\ref{prop1-arbitrary-point} to $K = \{e \in R(B_{X/Z}) \mycolon \Abs{e_1^*(e)} \le t\}$ we obtain some $0<\tau< 1$ such that $T_\tau (K) \subset U_E$. We also note for future use that
$ e_1^*(R(\widetilde x_1)) =  e_1^*(R(\widetilde x_2)) = t$, consequently  $R^*e_1^* =  \frac{t}{\alpha} \widetilde f_3$ and  $e_1^*(R(\widetilde x_3)) = \frac{t}{\alpha} > 1$.

Our goal is to demonstrate that $T = T_\tau \circ R \circ q \mycolon X \longrightarrow E$ is the operator we are looking for. Consider the composition $T_\tau \circ R \mycolon X/Z \longrightarrow E$. It is a bijection, which ensures that $\ker T = \ker q =  Z$.
The property $e_1^* \circ T_\tau = e_1^*$ implies that
$$
\|T_\tau \circ R\| \ge \Abs{ e_1^*((T_\tau \circ R)\widetilde x_3)} = e_1^*(R(\widetilde x_3)) > 1.
$$
Let $y \in S_{X/Z}$ be a point at which $\|(T_\tau \circ R) (y)\| = \|T_\tau \circ R\|$. Then $y$ must belong  to $\cone\{\widetilde x_1, \widetilde x_2\} \cup (-\cone\{\widetilde x_1, \widetilde x_2\})$ because
otherwise we would have $|\widetilde f_3(y)| \le \alpha$ and thus $|e_1^*(Ry)|= \frac t\alpha |\widetilde f_3(y)| \le t$ so that $(T_\tau\circ R)y \in T_\tau(K) \subset U_E$, which contradicts the estimate $\|(T_\tau \circ R)(y)\|>1$.

Replacing $y$ by $-y$, if necessary, we may assume that $y \in S_{X/Z}\cap \cone\{\widetilde x_1, \widetilde x_2\}$. Let $g \in S_{E^*}$ be a supporting functional for $[T_\tau \circ R](y)$, that is, $g \bigl([T_\tau \circ R](y)\bigr) = \|T_\tau \circ R\|$. Then,
$$
\frac{(T_\tau \circ R)^*g}{\|T_\tau \circ R\|}
$$
is a supporting functional at $y$, so, by Lemma~\ref{prop3-tangents-to-an-arc},
$$
\frac{(T_\tau \circ R)^*g}{\|T_\tau \circ R\|} \in \cone\{\widetilde f_1, \widetilde f_2\}
$$
and consequently
$$
q^*\left(\frac{(T_\tau \circ R)^*g}{\|T_\tau \circ R\|}\right) \in \cone\{f_1, f_2\}.
$$
Since $\cone\{f_1,f_2\}$ consists only of norm attaining functionals,
there is $x \in S_X$ such that
$$
\left[q^*\left(\frac{(T_\tau \circ R)^*g}{\|T_\tau \circ R\|}\right)\right](x) = 1.
$$
From this, we get that
\begin{align*}
\|T(x)\| &= \|(T_\tau \circ R \circ q)(x)\| \ge g((T_\tau \circ R \circ q)(x)) \\
&= \left(q^*\left((T_\tau \circ R)^*g\right)\right)(x) = \|T_\tau \circ R\|.
\end{align*}
On the other hand, $\|T\| \le \|T_\tau \circ R\|$, which means that $T$ attains its norm at $x$.

$\bullet$ {Case 2:} At least one of $\widetilde x_1$,  $\widetilde x_2$ is not a smooth point  of~$S_{X/Z}$. Without loss of generality we may assume that $\widetilde x_1$ is not a smooth point, so there are two linearly independent functionals $g_1, g_2 \in S_{(X/Z)^*}$ with $g_1(\widetilde x_1) = g_2(\widetilde x_1) = 1$.

Select a smooth point $e_1 \in S_E$, and let $e_1^* \in S_{E^*}$ be the supporting functional  at~$e_1$. Select $e_2 \in  S_E \cap \ker e_1^*$ and define $T_{\tau} \mycolon  E \longrightarrow E$  as in Lemma~\ref{prop1-arbitrary-point}. Denote $g = \frac12(g_1 + g_2)$ and choose $y \in S_{X/Z} \cap \ker g$. Denote by $R \mycolon X/Z \longrightarrow E$ the linear operator satisfying $R(\widetilde x_1) = e_1$ and $R(y) = e_2$. Then, $[R^*e_1^*](\widetilde x_1) = 1$, $[R^*e_1^*](y) = 0$, so $R^*e_1^* = g$. Let us consider those $h_1, h_2 \in E^*$ for which $R^*h_1 = g_1$ and  $R^*h_2 = g_2$. These   $h_1, h_2 \in E^*$  satisfy all the conditions of Lemma~\ref{prop2-smooth-point-and-an-angle}, so for the corresponding set
$$
A = \{x \in E \mycolon \max\{\Abs{h_1(x)}, \Abs{h_2(x)}\} \le 1\}
$$
there is $0<\tau<1$ such that $T_{\tau}(A) \subset B_E$.

Let us demonstrate that $T = T_\tau \circ R \circ q \mycolon X \longrightarrow E$ attains its norm at~$x_1$. Indeed,
\begin{align*}
T(B_X) &= T_{\tau}(R(q(B_X))) \subset T_{\tau}(R(B_{X/Z}))\\
&\subset T_{\tau}(R(\{\widetilde x \in X/Z \mycolon \max\{\Abs{g_1(\widetilde x)}, \Abs{g_2(\widetilde x)}\} \le 1\}))\\
&= T_{\tau}(A) \subset B_E,
\end{align*}
which gives us that $\|T\| \le 1$. But, on the other hand,
$$
T(x_1) = T_{\tau}(R(\widetilde x_1)) = T_{\tau}(e_1) = e_1,
$$
so $\|T(x_1)\| = 1$.
\end{proof}

Our last goal in this section is to present all the implications proved so far in the particular case of rank-two operators, and to discuss the possibility of reversing them.

Let $X$ and $Y$ be Banach spaces of dimension at least two, and let $Z$ be a closed subspace of $X$ of codimension two. Consider the following properties:
\begin{enumerate}
\item[(a)]
  $Z$ is the kernel of a norm-one projection.
\item[(b)]
  $Z$ is proximinal in $X$.
\item[(c)]
  $Z^\bot \subset \NA(X)$.
\item[(d)]
 Every $T\in \mathcal{L}(X,Y)$ with $\ker T\supset Z$ is norm attaining.
\item[(e)]
 Every $T\in \mathcal{L}(X,Y)$ with $\ker T=Z$ is norm attaining.
\item[($\diamondsuit$)]
There exists $T\in \mathcal{L}(X,Y)$ with $\ker T=Z$ which is norm attaining.
\item[(f)]
There are linearly independent $f,g\in Z^\bot$ such that $\cone\{f,g\}\subset Z^\bot\cap \NA(X)$.
\item[(g)] There are linearly independent $f,g\in Z^\bot\cap S_{X^*}$ and $x\in S_X$ such that $f(x)=1=g(x)$.
\end{enumerate}

Then, the following implications hold:
\begin{equation}\label{diageneral}
\begin{tikzpicture}[baseline=(m.center)]
  \matrix (m) [matrix of math nodes,row sep=2em,column sep=3.3em,minimum width=2.2em]
  {
     \text{(a)} & \text{(b)} & \text{(c)} & \text{(e)} & \fbox{\text{($\diamondsuit$)}} & \text{(f)} & \text{(g)} \\ & & \text{(d)}   \\};
  \path[-stealth]
    (m-1-1) edge [double]  (m-1-2)
    (m-1-2) edge [double] (m-1-3)
    (m-1-3) edge [double] node [above] {($\star$)} (m-1-4)
            edge [implies-implies,double]  (m-2-3)
     (m-1-4) edge [double] node [above] {($\star\star$)} (m-1-5)
      (m-1-6) edge [double] node [above] {($\star{\star}\star$)} (m-1-5)
       (m-1-7) edge [double] (m-1-6)
    ;
\end{tikzpicture}
\end{equation}

We now discuss these implications and  the possibility of  reversing them.

It is immediate that (a) implies (b), but the converse result is obviously false, even for finite-dimensional (non-Hilbertian) spaces. It is well known that (b) implies (c), but not conversely, see \cite[Section~2]{InduPLMS1982} or \cite[Section~2]{Rmoutil} for a discussion of this. The implication (c) $\Rightarrow$ (d) is Proposition~\ref{prop:kerbotinNAimpliesNA}, and the reverse implication is obvious using rank-one operators. Next, the implications (d) $\Rightarrow$ (e) $\Rightarrow$ ($\diamondsuit$) are obvious.

On the other side of condition ($\diamondsuit$), we have that (g) $\Rightarrow$ (f) since, obviously, $\cone\{f,g\}\subset \NA(X)$ if (g) holds, but the converse result is not true as follows by taking $X$ to be a smooth reflexive Banach space. That (f) implies ($\diamondsuit$) is exactly our Theorem~\ref{Theo2-implication}.

So it remains to discuss the possible converses of the implications ($\star$), ($\star\star$), and ($\star{\star}\star$).

Let us start by discussing  the possibility of the reciprocal result to implication ($\star$) above to be true. We have two different behaviors, depending on whether the range space is strictly convex or not (that is, whether the unit sphere of the range space does not or does contain non-trivial segments).

For non-strictly convex range spaces, we have the following positive result.

\begin{proposition}\label{prop:allNAimpliesybotNA-notrotund}
Let $X$ be a Banach space, let $E$ be a two-dimensional space which is not strictly convex and let $Z$ be a two-codimensional closed subspace of $X$. If every $T\in \mathcal{L}(X,E)$ with $\ker T=Z$ attains its norm, then $Z^\bot\subset \NA(X)$.
\end{proposition}

\begin{proof}
Let us start with the simpler case of $E=\ell_\infty^2$. Fix $\varphi\in Z^\bot$ with $\|\varphi\|=1$;   our aim is to show that $\varphi\in \NA(X)$. To get this, consider $\psi\in Z^\bot$ with $\|\psi\|=1$ such that $Z^\bot=\Lin\{\varphi,\psi\}$ and define $T\mycolon X\longrightarrow \ell_\infty^2$ by $Tx=\bigl(\varphi(x),\tfrac12 \psi(x)\bigr)$ for all $x\in X$. Then, $\|T\|=1$ and $\ker T=Z$, so $T\in \NA(X,\ell_\infty^2)$ by hypothesis. But then, clearly, $\varphi\in \NA(X)$, as desired.

Now, suppose that $E$ is a two-dimensional non-strictly convex space. Then, we may find a bijective norm-one operator $U\mycolon \ell_\infty^2 \longrightarrow E$ such that $U(1,t)\in S_E$ for every $t\in [-1,1]$ (we just have to use the segment contained in $S_E$). Now, fix $\varphi\in Z^\bot$ with $\|\varphi\|=1$, our aim is to show that $\varphi\in \NA(X)$. Again, we consider $\psi\in Z^\bot$ with $\|\psi\|=1$ such that $Z^\bot=\Lin\{\varphi,\psi\}$ and this time we define the operator $T\mycolon X \longrightarrow E$ by $Tx=U(\varphi(x),\tfrac12 \psi(x))$ for all $x\in X$. On the one hand, $\|T\|=1$: consider a sequence $\{x_n\}_{n\in \N}$ in $S_X$ such that $\varphi(x_n)\longrightarrow 1$ and, passing to a subsequence, we also have that $\psi(x_n)\longrightarrow t_0\in [-1,1]$ and so $$\|Tx_n\|=\|U(\varphi(x_n),\tfrac12 \psi(x_n)\|\longrightarrow \|U(1,t_0)\|=1.$$
On the other hand, $\ker T=Z$, so $T$ attains its norm by hypothesis. That is, there is $x\in S_X$ such that $1=\|Tx\|=\|U(\varphi(x),\tfrac12 \psi(x))\|$. But as $\|U\|=1$, this implies that $\|(\varphi(x),\tfrac12 \psi(x))\|_\infty=1$ and this immediately gives that $|\varphi(x)|=1$, that is, $\varphi\in \NA(X)$, as desired.
\end{proof}

When the range space is strictly convex, the above proof is not valid and, actually, the result is false as the following counterexample shows.

\begin{example}\label{example-allTNA-butYbotNOTNA-rotund}
Let $E$ be a two-dimensional strictly convex Banach space. Then, there are a Banach space $X$ and a two-codimensional closed subspace $Z$ of $X$ satisfying that every operator $T\in \mathcal{L}(X,E)$ with $\ker T=Z$ attains its norm, but $Z^\bot$ is not contained in $\NA(X)$.
\end{example}

\begin{proof}
Take a two-dimensional Banach space $W$ whose unit sphere $S_W$ contains a segment $[a, b]$, where the endpoints $a,b$ are extreme points of the sphere, the number of extreme points is countable, and the endpoints $a, b$ are smooth point of the sphere. Let $X=\ell_1$ and define an operator $U\mycolon X = \ell_1 \longrightarrow W$ that maps the vectors of the unit basis $\{e_n\}$ onto all the extreme points of $S_W$ with the exception of $\pm a$ and $\pm b$. Then,
$Z = \ker U$ is a two-codimensional closed subspace of $X$, whose annihilator $Z^\bot$ is not contained in $\NA(X)$, because if one takes $f \in W^*$ which attains its norm on $[a, b]$, then $U^*f\in Z^\bot$ does not attain its norm. On the other hand, as $\overline{U(B_X)}=B_W$, $X/Z$ is isometrically isomorphic to $W$ by virtue of the injectivization $\widetilde{U} \in \mathcal{L}(X/Z, W)$ of $U$ which satisfies $U=\widetilde{U}q$. So, if one takes an arbitrary norm-one operator $T\mycolon X \longrightarrow E$ with $\ker T=Z$, then $T$ factors through $U$ (or, what is the same, through the composition of the quotient map $q$ and $\widetilde{U}$). That is, $T = \widetilde T U$ for some norm-one operator $\widetilde{T}\mycolon W\longrightarrow E$, so the image $T(B_{X})$ of the closed unit ball is a linear copy (under~$\widetilde T$) of $U(B_{X}) = B_W \setminus ([a, b] \cup [-a, -b])$.

We shall argue that $\|\widetilde T(a)\|\neq 1$. Otherwise one could pick some $e^*\in S_{E^*}$ with $e^*(\widetilde T(a))=1$. It follows that $w^*:= \widetilde T^* (e^*)$ is a supporting functional at $a$ of norm one. Any supporting functional at $\frac12 (a+b)$ also supports~$a$, and by smoothness of~$a$ it has to coincide with~$w^*$. Consequently $w^*(a)= w^*(\frac12 (a+b)) = w^*(b)=1$ and so $\widetilde T(a)$, $\widetilde T(\frac12(a+b))$ and $\widetilde T(b)$ lie on a non-trivial segment of $S_E$, which is impossible when $E$ is strictly convex. Likewise $\|\widetilde T(b)\|\neq 1$. Hence there exists an extreme point $w$ of $B_W$ different from $\pm a$, $\pm b$ for which $\|\widetilde T(w)\| = 1$. This $w$ is of the form $w=U(e_n)$ for some~$n$, and we see that $T=\widetilde T U$ attains its norm at this~$e_n$.
\end{proof}

We would like to emphasize a question related to the example above, which asks about the possibility of  the implication (e) $\Rightarrow$ (f) being true when the range space is strictly convex (it is true for non-strictly convex range spaces by Proposition~\ref{prop:allNAimpliesybotNA-notrotund}). The failure of (f) $\Rightarrow$ (e) will be shown shortly in Example~\ref{example-exists-NA2-but-not-all}.

\begin{problem}
Let $X$ be a Banach space, let $Z$ be a closed subspace of $X$ of codimension two and let $E$ be a two-dimensional strictly convex space. Suppose that every $T\in \mathcal{L}(X,E)$ with $\ker T=Z$ attains its norm, does then $Z^\bot\cap \NA(X)$ contain a non-trivial cone?
\end{problem}

To show the failure of the converse implication to ($\star\star$) in diagram~(\ref{diageneral}), the next example works. Note that it also shows that (f) does not imply~(e).

\begin{example}\label{example-exists-NA2-but-not-all}
{\slshape There exists a rank-two operator $T\in \mathcal{L}(\ell_1,\ell_2^2)$ such that $[\ker T]^\bot \cap \NA(\ell_1)$ contains a non-trivial cone, but $T$ does not attain its norm.}

Indeed, let $T\in \mathcal{L}(\ell_1,\ell_2^2)$ be  an operator such that
$$
T(B_{\ell_1}) = \co\{\pm u_1, \pm \tfrac12 u_2\}\setminus \{\pm u_1\}
$$
where $\{u_1,u_2\}$ is the canonical basis of $\ell_2^2$. This operator can  easily be constructed by mapping the unit vector basis of $\ell_1$ onto a  countable dense subset of the union of the half-open segments
$(-u_1, \frac12u_2] \cup (u_1, \frac12u_2] \subset \ell_2^2$. Then, $\|T\| = 1$, but the norm is not attained. Nevertheless, the functionals from the cone in $(\ell_2^2)^*$ generated by $2u_2^* \pm u_1^*$  attain their maxima on $T(B_{\ell_1})$ at the point $\frac12 u_2$, so the image of this cone under the isomorphic embedding $T^*$ is contained in $[\ker T]^\bot$ and consists of norm attaining functionals.
\end{example}

Finally, it  follows from the next example
that the converse implication to $(\star{\star}\star)$ in diagram \eqref{diageneral} fails in general.

\begin{example}
{\slshape Let $X=\ell_1$ and let $E$ be an arbitrary two-dimensional space. Then, there is $T\in \NA^{(2)}(X,E)$ such that $[\ker T]^\bot\cap \NA(X)$ does not contain non-trivial cones.}

Indeed, let $\{z_n\mycolon n\geq 2\}$ be a dense subset of the open unit ball of $E$, let $u_0$ be a smooth point of $S_E$ whose unique support functional is called $u_0^*\in S_{E^*}$, and define $T\mycolon \ell_1 \longrightarrow E$ by means of the unit vector basis $\{e_n\}$ of $\ell_1$ by
$$
T(e_1)= u_0, \qquad T(e_n) = z_n \mbox{ for }n\ge2.
$$
Clearly, $T$ is onto and attains its norm (at $e_1$), so $T\in \NA^{(2)}(X,E)$. On the other hand, since $T(B_{\ell_1})$ is dense in $B_E$, the adjoint operator $T^*\mycolon E^* \longrightarrow (\ell_1)^*$ is an isometric embedding. Also, $T^*(E^*) \subset [\ker T]^\bot$ and the dimensions of both subspaces are equal to~$2$, so we have $[\ker T]^\bot = T^*(E^*)$. Now, consider an arbitrary non-zero $h \in [\ker T]^\bot\cap \NA(X)$. Let us write $h = T^*y^*$, where $y^* \in E^*$, and let $x = (x_1, x_2,\dots) \in B_{\ell_1}$ be such that $\|h\| = h(x)$. Then
$$
\|y^*\| = \|h\| =  h(x) = y^*(Tx) \le \|y^*\| \|Tx\| \le \|y^*\|,
$$
so $\|Tx\| = 1$ and $y^*$ attains its norm at $Tx$.
Taking into account that $\|x\| = \sum_{n = 1}^\infty  |x_n| = 1$, that $\|z_n\| < 1$ and that
$$
\|Tx\| \le |x_1| + \sum_{n \ge 2}  |x_n| \|z_n\|,
$$
we see that the equality $\|Tx\| = 1$ may happen only if $x = x_1 e_1$ with $|x_1| = 1$. Therefore, $Tx = \pm u_0$  and $y^*$ attains its norm at $Tx$, so $y^*$ is proportional to $u_0^*$. We have demonstrated that $[\ker T]^\bot \cap \NA(X) \subset \Lin T^*(u_0^*)$, so $\NA(X)$ does not contain two linearly independent elements of $[\ker T]^\bot$.
\end{example}

\section{Density of norm attaining finite rank operators}\label{section:density}

After discussing the existence of norm attaining finite rank operators, it is now time to study positive results for the density  of such operators.
An easy observation is pertinent, namely, we may restrict ourselves to consider finite-dimensional codomain spaces if we are interested in results valid for all codomain spaces:  if $X$ has the property that for all $Y$, all finite rank operators $T\mycolon X\longrightarrow Y$ can be approximated by norm attaining operators, then all such $T$ can be approximated by norm attaining finite rank operators.
Indeed,  if $T\mycolon X\longrightarrow Y$ has finite rank, then we may view $T\mycolon X\longrightarrow T(X)$, approximate $T$ here, and then compose the approximating sequence with the isometric inclusion operator from $T(X)$ into~$Y$.

The leading question here is the following open problem.

\begin{problem}
Is it true that every finite rank operator can be approximated by norm attaining (finite rank) operators?
\end{problem}

As in the previous section, we will focus on the domain spaces. So, the general aim in this section is to provide partial answers to the following question.

\begin{problem}\label{problem-sufficient-density}
Find sufficient conditions on a Banach space $X$ so that every finite rank operator whose domain is $X$ can be approximated by (finite rank) norm attaining operators.
\end{problem}

First, it is immediate that Lindenstrauss's property~(A) on a Banach space $X$ implies that a finite rank operator whose domain is $X$ can be approximated by norm attaining finite rank operators (just restrict the codomain to the range space, use property~(A) there and inject the range space again into the codomain). Therefore, some positive solutions to the problem above are the known sufficient conditions for property~(A) like the Radon-Nikod\'{y}m property, the property alpha, or the fact that the unit ball contains a set of uniformly strongly exposed points which generates the ball by taking the closed convex hull. We refer to the already cited survey paper \cite{Aco} for more information.

If one looks for less restrictive conditions valid for finite rank operators but not necessarily for all operators, there are  such conditions for compact operators. A detailed account of these properties is given in the survey paper~\cite{M-RACSAM}. But all of the known results of this kind need some sort of approximation property of the dual space, since they actually provide that every compact operator can be approximated by norm attaining finite rank operators.

Our main aim here is to try to provide a sufficient condition for the density of norm attaining finite rank operators which does not require the approximation property of the dual of the domain space.
Here is the result which follows directly from Proposition~\ref{prop:kerbotinNAimpliesNA}.

\begin{theorem}\label{thm:suf-density}
Let $X$ be a Banach space satisfying that for every $n\in \N$, every $\eps>0$ and all $x_1^*,\ldots,x_n^*\in B_{X^*}$, there are $y_1^*,\ldots,y_n^*\in B_{X^*}$ such that $\|x_i^*-y_i^*\|<\eps$ for every $i=1,\ldots,n$ and
$$
\Lin\{y_1^*,\ldots,y_n^*\}\subset \NA(X).
$$
Then, every finite rank operator whose domain is $X$ can be approximated by finite rank norm attaining operators.

If, moreover, $X^*$ has the approximation property, then every compact operator whose domain is $X$ can be approximated by finite rank norm attaining operators.
\end{theorem}

Before providing the proof of the theorem, let us state the main consequence which follows immediately from it.

\begin{corollary}\label{coro:denselineability=>NA(X,F)dense}
Let $X$ be a Banach space such that there is a norm dense linear subspace
of $X^*$ contained in $\NA(X)$. Then, for every Banach space $Y$, every operator $T\in \mathcal{L}(X,Y)$ of finite rank can be approximated by finite rank norm attaining operators.

If, moreover, $X^*$ has the approximation property, then every compact operator whose domain is $X$ can be approximated by finite rank norm attaining operators.
\end{corollary}

We do not know whether there are Banach spaces satisfying the conditions of Theorem~\ref{thm:suf-density} but not the ones of Corollary~\ref{coro:denselineability=>NA(X,F)dense}.

We may now give the pending proof.

\begin{proof}[Proof of Theorem~\ref{thm:suf-density}]
It is enough to show that $\NA(X,F)$ is dense in $\mathcal{L}(X,F)$ for every finite-dimensional space $F$. We fix an arbitrary finite-dimensional Banach space $F$ and consider an Auerbach basis $\{e_1,\ldots,e_n\}$ of~$F$ \cite[Theorem~4.5]{FHHMZ} with biorthogonal functionals $\{e_1^*,\ldots,e_n^*\}$ in~$F^*$. Given a norm-one operator $T\in \mathcal{L}(X,F)$ and $\eps>0$, let $x_i^*=T^*e_i^*\in B_{X^*}$ for $i=1,\ldots,n$, and observe that $T=\sum_{i=1}^n x_i^*\otimes e_i$. By hypothesis, we may find $y_1^*,\ldots,y_n^*\in B_{X^*}$ such that $\|x_i^*-y_i^*\|<\eps/n$ and $\Lin\{y_i^*,\ldots,y_n^*\}\subset \NA(X)$. We write $S=\sum_{i=1}^n y_i^*\otimes e_i\in \mathcal{L}(X,F)$ and first observe that $\|T-S\|<\eps$. On the other hand, as $S$ vanishes on $\bigcap_{i=1}^n \ker y_i^*$ we have that $[\ker S]^\bot \subset \Lin\{y_i^*,\ldots,y_n^*\}\subset \NA(X)$. This gives that $S\in \NA(X,F)$ by Proposition~\ref{prop:kerbotinNAimpliesNA}.

Let us show the moreover part: if $X^*$ has the approximation property, then every compact operator whose domain is $X$ can be approximated by finite rank operators (see \cite[Theorem~1.e.5]{LinTza-01} for instance) and the result now follows from the first part of the proof.
\end{proof}

As a consequence of this result, we may recover  some results stated in \cite{JoWo} and \cite[Section~3]{M-RACSAM} on norm attaining compact operators. The main tool provided in \cite{JoWo} to get solutions to Problem~\ref{problem-sufficient-density} is the following easy observation.

\begin{corollary}[{\rm \cite[Lemma~3.1]{JoWo}}] \label{corollary-Lemma31JohnsonWolfe}
Let $X$ be a Banach space such that for all $x_1^*,\ldots,x_n^*\in B_{X^*}$ and every $\eps>0$, there is a norm-one projection $P\in \mathcal{L}(X,X)$ of finite rank such that $\max_i \|x_i^*-P^*(x_i^*)\|<\eps$. Then, every compact operator whose domain is $X$ can be approximated by norm attaining operators of finite rank.
\end{corollary}

This result can also easily be deduced  from Theorem~\ref{thm:suf-density} as the hypotheses imply that $X^*$ has the approximation property and that the subspace $P^*(X^*)$ is contained in $\NA(X)$ (indeed, $[P^*(x^*)](B_X)=x^*(P(B_X))$ is compact as $P$ is a finite rank projection, and $P(B_X)= B_{P(X)}$ since $\|P\|=1$).

This result applies to $X=C_0(L)$ for every locally compact Hausdorff space $L$ \cite[Proposition~3.2]{JoWo} and also to $L_1(\mu)$ for every finite positive measure $\mu$ (see \cite[Lemma~3.12]{Dantas} for a detailed proof). For $C_0(L)$, we do not know whether it is actually true that $\NA(C_0(L))$ contains a dense linear subspace. In the case of $L_1(\mu)$ for a localizable measure $\mu$ (see e.g.\ \cite[Def.~211G]{Fremlin2} for the definition), the subspace of $L_1(\mu)^*=L_\infty(\mu)$ of those functions in $L_\infty(\mu)$ taking finitely many values, i.e., the subspace of step functions, is clearly contained in $\NA(L_1(\mu))$ and it is dense in~$L_\infty(\mu)$. Of course, the hypothesis of being localizable may be dropped, as every $L_1(\mu)$ space is isometrically isomorphic to an $L_1(\nu)$-space where $\nu$ is localizable (this follows, for instance, by Maharam's theorem). Let us comment that the fact that norm attaining compact operators from an $L_1(\mu)$ space are dense in the space of compact operators was proved in \cite[p.~6]{Diestel-Uhl-Rocky}.

Let us state these two results.

\begin{corollary}[\mbox{\rm \cite[Proposition~3.2]{JoWo}}]
Let $L$ be a locally compact Hausdorff topological space. Then, $X=C_0(L)$ satisfies the hypotheses of Theorem~\ref{thm:suf-density}. Therefore, every compact linear operator whose domain is $C_0(L)$ can be approximated by finite rank norm attaining operators.
\end{corollary}

\begin{corollary}[\mbox{\rm extension of \cite[p.~6]{Diestel-Uhl-Rocky}}]
Let $\mu$ be a positive measure. Then, there is a dense linear subspace of $L_1(\mu)^*$ which is contained in $\NA(L_1(\mu))$. As a consequence, every compact linear operator whose domain is $L_1(\mu)$ can be approximated by finite rank norm attaining operators.
\end{corollary}

Another known case in which Corollary~\ref{corollary-Lemma31JohnsonWolfe} applies is the case of isometric preduals of $\ell_1$ \cite[Corollary~3.8]{M-RACSAM}. Here, we are also able to get dense lineability of the set of norm attaining functionals.

\begin{corollary}[\mbox{\rm extension of \cite[Corollary~3.8]{M-RACSAM}}]
Let $X$ be an isometric predual of $\ell_1$. Then, there is a norm dense linear subspace of $X^*$ contained in $\NA(X)$. Therefore, for every Banach space $Y$, every compact operator $T\in \mathcal{L}(X,Y)$ can be approximated by finite rank norm attaining operators.
\end{corollary}

\begin{proof}
We just have to justify the existence of  a dense linear subspace of $X^*$ contained in $\NA(X)$, the rest of the results follows from Corollary~\ref{coro:denselineability=>NA(X,F)dense}. Indeed, it is shown in the proof of \cite[Corollary~3.8]{M-RACSAM} (based on results by Gasparis from 2002) that there is a sequence of finite rank norm-one projections $Q_n\in \mathcal{L}(X,X)$ such that the sequence $\{Q_n^*\}_{n\in \N}$ has increasing ranges and converges pointwise to the identity of~$X^*$. Then, $\bigcup\nolimits_{n\in \N} Q_n^*(X^*)$ is a subspace contained in $\NA(X)$ since each $Q_n$ is a norm-one projection of finite rank; this subspace is dense   by the pointwise convergence of $\{Q_n^*\}$ to the identity.
\end{proof}

Another easy case in which Corollary~\ref{corollary-Lemma31JohnsonWolfe} applies is when a Banach space $X$ has a shrinking monotone Schauder basis \cite[Corollary~3.10]{M-RACSAM} but, actually, the result follows from  Corollary~\ref{coro:denselineability=>NA(X,F)dense} as $\NA(X)$ contains a dense linear subspace in this case \cite[Theorem~3.1]{AcoAizAronGar}.

\begin{corollary}[\mbox{\rm \cite[Corollary~3.10]{M-RACSAM} and \cite[Theorem~3.1]{AcoAizAronGar}}]
Let $X$ be a Banach space. If $X$ has a shrinking monotone Schauder basis, then $\NA(X)$ contains a dense linear subspace. Therefore, every compact operator whose domain is $X$ can be approximated by norm attaining finite rank operators.
\end{corollary}

This applies, in particular, to closed subspaces of $c_0$ with a monotone Schauder basis, as shown in \cite[Corollary~12]{M-JFA2014} using a  result of G.~Godefroy and P.~Saphar from 1988.

\begin{example}[\mbox{\rm \cite[Corollary~12]{M-JFA2014}}]\label{example:subspace_c0_monotoneSchauder}
Let $X$ be a closed subspace of $c_0$ with a monotone Schauder basis. Then, $\NA(X)$ contains a dense linear subspace. Therefore, every compact operator whose domain is $X$ can be approximated by norm attaining finite rank operators.
\end{example}

Next we get the following result as an obvious consequence of Corollary~\ref{coro:denselineability=>NA(X,F)dense} (and the Bishop-Phelps theorem).

\begin{corollary}\label{coro:NAlinear}
Let $X$ be a Banach space. If $\NA(X)$ is a linear subspace of $X^*$, then finite rank operators with domain  $X$ can be approximated by finite rank norm attaining operators.

If, moreover, $X^*$ has the approximation property, then actually compact operators with domain  $X$ can be approximated by finite rank norm attaining operators.
\end{corollary}

Of course, the result above applies to $c_0$, but also when $X$ is a finite-codimensional proximinal subspace of $c_0$, as shown in \cite[Remark~b on p.~180]{GodIndu1999}. Besides, the non-commutative case also holds: $\NA(\mathcal{K}(\ell_2))$ is also a linear space (see \cite[Lemma]{GodIndPic}), so Corollary~\ref{coro:NAlinear} applies to it. Moreover, this linearity property of the set of norm attaining operators passes down to every finite-codimensional proximinal subspace of $\mathcal{K}(\ell_2)$, see \cite[Section~3]{GodIndu1999}. Finally, if $X$ is a $c_0$-sum of reflexive spaces, then clearly $\NA(X)$ is a linear subspace of $X^*$. Let us state all the examples we have presented so far.

\begin{examples}
{\slshape The following spaces satisfy that their sets of norm attaining functionals are vector spaces:
\begin{enumerate}
  \item[(a)] $c_0$ and its finite-codimensional proximinal subspaces;
  \item[(b)] $\mathcal{K}(\ell_2)$ and its finite-codimensional proximinal subspaces;
  \item[(c)] $c_0$-sums of reflexive spaces.
\end{enumerate}
Therefore, a finite rank operator whose domain is any of the spaces above can be approximated by norm attaining finite rank operators.}
\end{examples}

In the simplest case of closed subspaces of $c_0$, we do not know whether the hypothesis of finite codimension or the hypothesis of proximinality can be dropped in (a) above. What  is easy to show is that there is a closed hyperplane of $c_0$ whose set of norm attaining functionals is not a vector space (see \cite[Remark~b on p.~180]{GodIndu1999} again).

\begin{problem}
Let $X$ be a closed subspace of $c_0$. Is it true that every finite rank operator whose domain is $X$ can be approximated by norm attaining (finite rank) operators?
\end{problem}

Let us note that there are compact operators whose domains are closed subspaces of $c_0$ which cannot be approximated by norm attaining operators \cite[Proposition~3]{M-JFA2014}.

Corollary~\ref{corollary-Lemma31JohnsonWolfe} depends heavily on the fact that the norm of the projections is~$1$ and fails if one considers renormings. By contrast, Theorem~\ref{thm:suf-density} and its consequence Corollary~\ref{coro:denselineability=>NA(X,F)dense} only depend on  the set of norm attaining functionals itself, so both remain valid for renormings which conserve this set. In \cite[Theorem~9.(4)]{DebsGodSR}, it is shown that every separable Banach space $X$ admits a smooth renorming $\widetilde{X}$ such that $\NA(X)=\NA(\widetilde{X})$, and this result has recently been extended  to weakly compactly generated spaces (WCG spaces) \cite[Proposition~2.3]{GuiMonZiz}. Therefore, if Theorem~\ref{thm:suf-density} applies for a WCG space $X$, then so it does for the corresponding~$\widetilde{X}$. In particular, we get the following examples.

\begin{examples}
  {\slshape Let $X$ be a WCG Banach space which is equal to $C_0(L)$, is equal to $L_1(\mu)$, satisfies that $X^*=\ell_1$, or  is a finite-codimensional proximinal subspace of $c_0$ or of $\mathcal{K}(\ell_2)$.
(In the latter cases, $X$ is of course separable.) Let $\widetilde{X}$ be the equivalent smooth renorming of $X$  given in \cite[Proposition~2.3]{GuiMonZiz} such that $\NA(\widetilde{X})=\NA(X)$. Then, every compact operator whose domain is $\widetilde{X}$ can be approximated by norm attaining (for the norm of $\widetilde{X}$) finite rank operators.}
\end{examples}

We do not even know whether the particular case of $\widetilde{c_0}$ can be deduced from previously known results.

Although it is not directly related to finite rank operators, we would like to finish the section by providing a condition which extends the known result by Lindenstrauss \cite[Theorem~1]{Lin} that reflexive spaces have property~(A), that is, an operator whose domain is a reflexive space can be approximated by norm attaining operators (this fact is actually used in the proof below).

\begin{proposition}\label{prop:kerbotinNA-density}
Let $X, Y$ be Banach spaces. Then, every operator $T\in \mathcal{L}(X,Y)$ for which $[\ker T]^\bot\subset \NA(X)$ can be approximated by norm attaining operators (whose kernels contain $\ker T$).
\end{proposition}

\begin{proof}
We follow the lines of the proof of Proposition~\ref{prop:kerbotinNAimpliesNA}.
As $[\ker T]^\bot\subset \NA(X)$, it is immediate from James's theorem that $X/{\ker T}$ is reflexive (see the proof of \cite[Lemma~2.2]{Band-Godefroy}). Now, $T$ factors through $X/{\ker T}$, that is, there is an operator $\widetilde{T}\mycolon X/{\ker T} \longrightarrow Y$ such that $T=\widetilde{T}\circ q$, and it is clear that $\|\widetilde{T}\|=\|T\|$. By the result of J.~Lindenstrauss just mentioned, \cite[Theorem~1]{Lin}, there is a sequence $\bigl\{\widetilde{S}_n\bigr\}_{n\in \N}$ of norm attaining operators from $X/{\ker T}$ into $Y$ which converges in norm to $\widetilde{T}$. On the one hand, the same argument as the one given in Proposition~\ref{prop:kerbotinNAimpliesNA} allows us to see that for every $n\in \N$, the operator $S_n:= \widetilde{S}_n \circ q\mycolon X\longrightarrow Y$ attains its norm. On the other hand, it is clear that $\|S_n-T\|\leq \|\widetilde{S}_n - \widetilde{T}\|\longrightarrow 0$, so $\{S_n\}_{n\in \N}$ converges to~$T$.
\end{proof}

\section{Norm attaining operators onto a two-dimensional Hilbert space}\label{sect:naoperators}
Our aim in this section is to study the special case when the range space is a (two-dimensional) Hilbert space, where some specific tools can be used, for instance,  we may rotate every point of the unit sphere to any other one. As shown in Remark~\ref{remark:toproblemexistence}, this study is actually equivalent to the study of the existence of norm attaining operators of rank two into all Banach spaces of dimension greater than or equal to two.

We will eventually provide some characterizations of the fact that an operator from a Banach space onto a two-dimensional Hilbert space attains its norm and also a characterization of when norm attaining operators onto a two-dimensional Hilbert space are dense.

Let us observe that the existence of a norm attaining operator $T$ of rank at least~$2$ from a Banach space $X$ to a Hilbert space $H$ gives the existence of a surjective norm attaining operator from $X$ onto $\ell_2^2$ (just composing $T$ with a convenient orthogonal projection).

Let $T\mycolon X\longrightarrow \ell_2^2$ be an operator of rank two. One can identify $T$ with a pair of linearly independent functionals $(f,g)\in X^*\times X^*$. Throughout this section, we will make this identification without further reference. Note that, obviously, $\|(\sigma f, \sigma' g)\|\le \|(f,g)\|$ if $|\sigma|, |\sigma'|\le 1$ so, in particular,
$$
\|(\pm f, \pm g)\|=\|(f,g)\| \quad \text{ and } \quad \max\{\|f\|,\|g\|\} \le \|(f,g)\|.
$$
We will also use these facts frequently in this section without recalling them.

Our first goal in this section is to characterize when $\|(f,g)\|\le 1 $ in terms of the functionals $f$ and $g$, especially in the case when $\|f\|=1$ and $f$, $g$ are linearly independent. We next will use this idea to produce pairs of functionals of this form.

The desired characterization of when $\Vert(f,g)\Vert\leq 1$ is the following.

\begin{proposition}\label{estimation2}
  Let $X$ be a Banach space, and let $f\in S_{X^*}$ and $g\in B_{X^*}$ be linearly independent.
  Then, $\|(f,g)\|\le 1$ if and only if
  \begin{equation}\label{root}
  \Vert f+tg\Vert \leq \sqrt{1+t^2}\qquad\mbox{for all }t\in \R .
  \end{equation}
\end{proposition}

We need the following easy lemma.

\begin{lemma}\label{distance}
  Let $X$ be a Banach space. Fix $z_0\in X$ and linearly independent $f, g\in X^*$, and consider $M=\ker f\cap \ker g$. Then,
  $$
  \dist(z_0, M)=\sup_{(t, s)\in \R^2\setminus\{(0,0)\}}\frac{\vert t f(z_0)+ s g(z_0)\vert}{\Vert t f+ s g\Vert}.
  $$
\end{lemma}

\begin{proof}
  If we consider $z_0=J_X(z_0)$ as an element of $X^{**}$, we have that
  $$
  \dist(z_0,M)=\dist(J_X(z_0),J_X(M))=\Vert {J_X(z_0)}|_{M^{\bot}}\Vert,
  $$
  where ${J_X(z_0)}|_{M^{\bot}}$ denotes the restriction of $J_X(z_0)$ to the subspace  $M^{\bot}$ of $X^*$. But $M^{\bot}$ is the subspace of $X^*$ generated by $f$ and $g$, hence
  $$
  \dist(z_0,M)=
  \sup_{x^*\in M^{\bot}\setminus\{0\}}\frac{\vert x^*(z_0)\vert}{\Vert x^*\Vert }=
  \sup_{(t,s)\in \R^2\setminus\{(0,0)\}}\frac{\vert t f(z_0)+ s g(z_0)\vert}{\Vert t f+s g\Vert},
  $$
  and we are done.
\end{proof}

We can now give the pending proof.

\begin{proof}[Proof of Proposition~\ref{estimation2}]
    As $f$ and $g$ are linearly independent, there are $ x_0,x_1\in X$ such that $f(x_0)=1$, $g(x_0)=0$ and  $g(x_1)=1$, $f(x_1)=0$. We then  have  that $X=M\oplus \Lin\{x_0,x_1\}$ with $M=\ker f \cap \ker g$. Now, for  $ T :=(f,g)$, $\|T\|\le 1$ if,  and only if,
  $$
  \|(\lambda,\mu)\|_2=\Vert T(m+\lambda x_0+\mu x_1)\Vert_2\leq\Vert m+\lambda x_0+\mu x_1\Vert \qquad \mbox{for all }\lambda , \mu \in\R, \ m\in M .
  $$
The above is equivalent to
\begin{equation}   \label{normlessone}
\sqrt{\lambda^2+\mu^2}\leq \dist(\lambda x_0+\mu x_1, M)\qquad \mbox{for all } \lambda , \mu \in\R .
\end{equation}
If we define $\vert (\lambda ,\mu)\vert =\dist(\lambda x_0+\mu x_1, M)$ for all $\lambda,\mu\in \R$, we get a norm on~$\R^2$. Now, using  Lemma~\ref{distance}, we see that
$$
\vert (\lambda ,\mu)\vert= \sup_{(t,s)\in \R^2\setminus\{(0,0)\}}\frac{\vert\lambda t+\mu s\vert}{\Vert tf+sg\Vert}.
$$
We deduce that  the dual norm of the above norm is given by  $\vert(\lambda ,\mu) \vert^*=\Vert \lambda f+\mu g\Vert$ for all $\lambda,\mu \in \R$, since by definition $|\;.\;|$ is the dual norm of $|\;.\;|^*$.
Taking dual norms in the inequality (\ref{normlessone}), we get that this inequality is equivalent to
$$
\Vert \lambda f+\mu g\Vert\leq \sqrt{\lambda^2+\mu^2}\qquad  \mbox{for all } \lambda , \mu \in\R.
$$
The last inequality can be rephrased by saying  that $\Vert g\Vert\leq 1$ and
\[
\Vert f+tg\Vert \leq \sqrt{1+t^2}\qquad \mbox{for all }t\in \R.
\qedhere
\]
\end{proof}

\begin{remark}\label{rem2.3}
We observe that (\ref{root})  implies that $g(x_0)=0$ whenever $f(x_0)=1=\|x_0\|=\|f\|$.
\end{remark}

To facilitate the notation, we introduce the following vocabulary.

\begin{definition}
Given $f\in S_{X^*}$, we call an element $g\in B_{X^*}\setminus\{0\}$ such that $(f,g)$ has rank two and $\|(f,g)\|=1$, a \textit{mate} of~$f$. This is equivalent to requiring that $\Vert f+tg\Vert\leq \sqrt{1+t^2}$ for all $t\in \R$, by Proposition~\ref{estimation2}.
\end{definition}

Observe that if $g\in B_{X^*}\setminus \{0\}$ is a mate of $f\in S_{X^*}$, one has that
$$
\lim_{t\to 0}\frac{\Vert f+tg\Vert-1}{t}=0 \quad \mbox{and}\quad
\limsup_{t\to 0}\frac{\Vert f+tg\Vert-1}{t^2}\leq \frac12<\infty.
$$
The last condition suggests another formulation of the existence of mates, which will be shown next.

\begin{proposition}\label{Limsup}
  Let $X$ be a Banach space and $f\in S_{X^*}$. Then $f$ has a mate
  if and only if there exist $h\in B_{X^*}\setminus\{0\}$  and $K,\varepsilon>0$  such that
  $$
  \Vert f+th\Vert\leq 1+Kt^2\qquad \mbox{for all } t\in (-\varepsilon,\varepsilon),
  $$
  equivalently,
  $$
  \limsup_{t\to 0}\frac{\Vert f+th\Vert-1}{t^2}<\infty.
  $$
  In fact, given $f\in S_{X^*}$ and $h\in B_{X^*}\setminus\{0\}$ such that $\limsup\limits_{t\to 0}\frac{\Vert f+th\Vert-1}{t^2}<\infty$,  there exists $0<s\leq 1$ such that  $sh$ is a mate of~$f$.
\end{proposition}

\begin{proof}
  The proof of the necessity of the limsup condition is given in  the previous comment. For the sufficiency, assume that $h\in B_{X^*}\setminus\{0\}$, $K,\varepsilon>0$  are such that $\Vert f+th\Vert\leq 1+Kt^2$ for all $t\in (-\varepsilon,\varepsilon)$.
(Note that this implies that $f$ and $h$ are linearly independent.)
  It is enough to show that there exists $0<s\le 1$ such that $\Vert f+tsh\Vert\leq\sqrt{1+t^2}$ for all $t\in \R$. If not, there is a sequence $\{t_n\}$ in $\R$ such that
\begin{equation}\label{limsup}
\Bigl\Vert f+\frac{t_n}{n}h \Bigr\Vert>\sqrt{1+t_n^2} \qquad \mbox{for all }n\in\N.
\end{equation}
Now,
$$
1+\frac{\vert t_n\vert}{n}\Vert h\Vert \geq
\Bigl\Vert f+\frac{t_n}{n}h\Bigr\Vert > \sqrt{1+t_n^2}
$$
for all $n\in \N$, and we deduce that
$$
\frac{\vert t_n\vert}{n} < \frac{2\Vert h\Vert}{n^2-\Vert h\Vert^2}\leq\frac{2}{n^2-1}
$$
for $n>1$. Then, $t_n/n\longrightarrow 0$. From (\ref{limsup}) we get that
$$
\frac{\Vert f+\frac{t_n}{n}h\Vert^2-1}{t_n^2/n^2}\geq n^2
$$
and so
$$
\limsup_{t\to 0}\frac{\Vert f+th\Vert^2-1}{t^2}=+\infty.
$$
But $\Vert f+th\Vert^2-1=(\Vert f+th\Vert+1) (\Vert f+th\Vert-1)$
and  $\lim_{t\to 0}\Vert f+th\Vert+1=2$, and so
$$
\limsup_{t\to 0}\frac{\Vert f+th\Vert-1}{t^2}=+\infty.
$$
This completes the proof.
\end{proof}

Now, we can formulate a first positive result about the existence of mates.

\begin{lemma}\label{lem2.5}
Let $X$ be a Banach space. If $f\in S_{X^*}$ is not an extreme point of $B_{X^*}$, then $f$ has a mate.
\end{lemma}

\begin{proof}
  Suppose $f=\frac12 (f_1+f_2)$, with $f_j\in B_{X^*}$ and $g:= \frac12 (f_1-f_2)\neq0$. Clearly, $\|g\|\le 1$, and $f$ and $g$ are linearly independent. We shall show that $g$ is a mate of $f$ using Proposition~\ref{estimation2}. For $t\in \R$ we have
  $$
  \|f+tg\| = \Bigl\| \frac12 f_1 + \frac12 f_2 + \frac t2 f_2 - \frac t2 f_1\Bigr\| =
  \Bigl\| \frac {1-t}2 f_1 + \frac{1+t}2 f_2 \Bigr\|.
  $$
  The latter norm is $\le1$ for $|t| \le 1$ and is $\le \frac {|t|-1}2  + \frac{1+|t|}2 = |t|$ if $|t|\ge1$. In either case we have $\|f+tg\| \le \sqrt{1+t^2}$.
\end{proof}

As an immediate consequence, if $X^*$ is not strictly convex, then there is $f\in S_{X^*}$ with a mate (this is a not very surprising result, see Proposition~\ref{prop2.8}). If actually $X$ is not smooth, then we get a more interesting result.

\begin{cor}\label{nonsmooth}
Let $X$ be non-smooth Banach space. Then, there is $f\in\NA_1(X)$ with a mate.
\end{cor}

\begin{proof}
Suppose $x\in S_X$ is such that there are distinct $f_1,f_2\in S_{X^*}$ with $f_1(x)=f_2(x)=1$. Then $f:=\frac12 (f_1+f_2)$ has norm~$1$ and attains its norm at~$x$, but is not an extreme point of the dual unit ball and Lemma~\ref{lem2.5} applies.
\end{proof}

We may also provide a characterization of mates in terms of extreme points of the space of operators.

\begin{proposition}\label{extreme}
Let $X$ be a Banach space and $f\in S_{X^*}$. There is a mate $g\in B_{X^*}\setminus\{0\}$ for $f$ if, and only if, the operator $(f,0)$ is not an extreme point of~$B_{\mathcal{L}(X,\ell_2^2)}$.
\end{proposition}

\begin{proof}
Suppose $g\neq0$ and $\|(f,g)\| \le1$. Then also $\|(f,-g)\|\le1$, and so $(f,0)=\frac12 \bigl((f,g)+(f,-g)\bigr)$ is not an extreme point of the unit ball of $\mathcal{L}(X,\ell_2^2)$.

Conversely, if $(f,0)$ is not an extreme point of the unit ball of $\mathcal{L}(X,\ell_2^2)$, then there is a nontrivial convex combination in the unit ball of $\mathcal{L}(X,\ell_2^2)$ representing $(f,0)$, say $(f,0)= \frac12 \bigl((f_1,g_1) + (f_2,g_2)\bigr)$ where $(f,0)\neq (f_1,g_1)$. If $f_1=f_2=f$, then necessarily $g_1\neq0$ and hence $\|(f,g_1)\|\le1$ and $f$ and $g_1$ are linearly independent; if not, then $f$ is not an extreme point of the unit ball, and hence has a mate by Lemma~\ref{lem2.5}.
\end{proof}

We can also ask  if there exists some Banach space $X$ of dimension at least two such that there is no mate for any element in $S_{X^*}$, equivalently $(f,0)$ is an extreme point of $B_{\mathcal{L}(X,\ell_2^2)}$ for every $f\in S_{X^*}$. From Lemma~\ref{lem2.5} we know that the dual of such an example cannot be strictly convex. Indeed, there is no such space whatsoever.

\begin{proposition}\label{prop2.8}
  Let $X$ be a Banach space with $\dim(X)\ge2$. Then, there exists $f\in S_{X^*}$ with a mate. Actually, given linearly independent $f',g'\in X^*$  such that $\|(f',g')\|=1$, there is a rotation $\pi$ on $\ell_2^2$ such that $(f',g')=\pi\circ (f,g)$, $f\in S_{X^*}$ and $g$ is a mate for $f$.
\end{proposition}

\begin{proof}
  Consider the rank two operator $T=(f',g')\in \mathcal{L}^{(2)}(X,\ell_2^2)$ with $\|T\|=1$. Then, $T^*\in \NA(\ell_2^2,X^*)$ and so $T^{**}\in \NA(X^{**},\ell_2^2)$, so there is $x_0^{**}\in S_{X^{**}}$ such that $\Vert T^{**}(x_0^{**})\Vert =\Vert x_0^{**}\Vert=1$. Now, we compose $T$ with a rotation $\pi'$ on $\ell_2^2$ to get a new operator $S=\pi' T$ with $\|S\|=\|S^{**}(x_0^{**})\|=1$ and $S^{**}(x_0^{**})=(1,0)$. Of course, $S$ still has rank two and is represented by
  $$
  (f,g)= \bigl( \cos(\varphi) \cdot f' + \sin(\varphi) \cdot g',\
  -\sin(\varphi) \cdot f' + \cos(\varphi) \cdot g'\bigr)
  $$
  for suitable $\varphi\in (-\pi,\pi]$. Then, we have that $f\in B_{X^*}$, $g\in B_{X^*}\setminus\{0\}$ satisfy that $x_0^{**}(f)=1$ and $x_0^{**}(g)=0$. Therefore $\|f\|=1$, $(f,g)$ has rank two, and $\|(f,g)\|=1$. That is, $g$ is a mate for~$f$.
\end{proof}

We now use all the previous ideas to study norm attaining operators. First, the next result explains the link between norm attaining operators and the existence of mates. It says that, up to rotation and rescaling, norm attaining operators onto $\ell_2^2$ are pairs of the form $(f,g)$ where $f\in \NA_1(X)$ and $g$ is a mate of $f$.

\begin{theorem}\label{0}
Let $X$ be a Banach space and let $T\in \mathcal{L}^{(2)}(X,\ell_2^2)$ with $\|T\|=1$. Then, the following assertions are equivalent:
\begin{enumerate}
  \item[(i)] $T\in \NA^{(2)}(X,\ell_2^2)$.
  \item[(ii)] There are $f, g\in X^*\setminus \{0\}$, $x_0\in S_X$ and a rotation $\pi$ on $\ell_2^2$ such that $f\in \NA_1(X)$ with $f(x_0)=1$, $g(x_0)=0$, $\Vert (f,g)\Vert\leq 1$, and $T=\pi\circ (f,g)$.
  \item[(iii)] There are $f\in \NA_1(X)$ with a mate $g\in B_{X^*}\setminus \{0\}$ and a rotation on $\ell_2^2$ such that $T=\pi\circ (f,g)$.
\end{enumerate}
\end{theorem}

\begin{proof}
(i) $\Rightarrow$ (ii). Suppose $T=(f',g')\in \NA^{(2)}(X,\ell_2^2)$, say $\|T\|=\|T(x_0)\|=1$ for some $x_0\in S_X$.
Using a rotation as in the proof of Proposition~\ref{prop2.8}, we
get a new operator $S=\pi T$ with $\|S\|=\|S(x_0)\|=1$ and $S(x_0)=(1,0)$. Now, $S=(f,g)$
satisfies $f(x_0)=1$, $g(x_0)=0$. Hence $\|f\|=1$ and $f\in \NA_1(X)$, but $g\neq0$ since $S$ has rank two as well, that is, $g$ is a mate for $f$. The converse implication (ii) $\Rightarrow$ (i) is clear as $S=(f,g)$ attains its norm at $x_0$ and so does $T=\pi S$.

Finally, (ii) $\Rightarrow$ (iii) is immediate and (iii) $\Rightarrow$ (ii) follows from Remark~\ref{rem2.3}.
\end{proof}

The following corollary summarizes the results of this section so far.

\begin{corollary}\label{prop2.9}
Let $X$ be a Banach space with $\dim(X)\geq 2$. Then the following assertions are equivalent:
  \begin{enumerate}
  \item[(i)]
    $\NA^{(2)}(X,\ell_2^2)\neq \emptyset$.
  \item[(ii)] There is $f\in \NA_1(X)$ with a mate.
  \item[(iii)]
    There are $f\in \NA_1(X)$ and $g\in B_{X^*}\setminus\{0\}$ such that $\Vert f+tg\Vert\leq \sqrt{1+t^2}$ for all $t\in\R$.
  \item[(iv)]
    There are $f\in \NA_1(X)$, $g\in  B_{X^*}\setminus\{0\}$, and $\varepsilon>0$ such that $\Vert f+tg\Vert\leq 1+\frac{t^2}{2}$ for all $t\in (-\varepsilon,\varepsilon)$.
  \item[(v)]
    There are $f\in \NA_1(X)$, $h\in  B_{X^*}\setminus\{0\}$, and $\varepsilon,  K>0$ such that $\Vert f+th\Vert\leq 1+Kt^2$ for all $t\in (-\varepsilon,\varepsilon)$.
  \item[(vi)] There are  $f\in \NA_1(X)$ and $h\in  B_{X^*}\setminus\{0\}$ such that
  $
  \displaystyle\limsup_{t\to 0}\frac{\Vert f+th\Vert-1}{t^2}<\infty.
  $
  \item[(vii)]
    There is $f\in \NA_1(X)$ such that $(f,0)$ is not an extreme point in the unit ball of $\mathcal{L}(X,\ell_2^2)$.
  \end{enumerate}
  The above conditions hold automatically if $X$ is non-smooth.
\end{corollary}

\begin{proof}
The equivalence between (i) and (ii) follows from Theorem~\ref{0}. The implication (ii) $\Rightarrow$ (iii) is Proposition~\ref{estimation2} and the implications (iii) $\Rightarrow $ (iv)  $\Rightarrow $ (v) $\Rightarrow $ (vi) are trivial. The implication (vi) $\Rightarrow$ (ii) is Proposition~\ref{Limsup}. Finally, the equivalence between (ii) and (vii) is Proposition~\ref{extreme}.

The validity of the conditions in the non-smooth case is remarked in Proposition~\ref{nonsmooth}.
\end{proof}

We note from Corollary~\ref{prop2.9}.vii that $\NA^{(2)}(X,\ell_2^2)=\emptyset$ if, and only if, $(f,0)$ is an extreme point of $B_{\mathcal{L}(X,\ell_2^2)}$ for every $f\in \NA_1(X)$, which implies that every $f\in \NA_1(X)$ is an extreme point of $B_{X^*}$. Again, we see that if $X$ is not smooth there are norm attaining operators from $X$ onto $\ell_2^2$.

The proof of Proposition~\ref{prop2.8} implies the following positive result. We already know the result from Proposition~\ref{prop:kerbotinNAimpliesNA} (or even from Theorem~\ref{Theo2-implication} which shows that it is valid even with a weaker hypothesis), but we include this alternative proof here for completeness.

\begin{cor}
Suppose $X$ is a Banach space for which $\NA(X)$ contains a two-dimensional subspace. Then, $\NA^{(2)}(X,\ell_2^2)\neq \emptyset$.
\end{cor}

\begin{proof}
  Suppose $f'$ and $g'$ are linearly independent so that $\Lin\{f',g'\}\subset \NA(X)$ and $\|(f',g')\|=1$.
  It was shown in the proof of Proposition~\ref{prop2.8} how to obtain some $f\in S_{X^*}$ with a mate by performing a rotation; note that this $f$ is a linear combination of $f'$ and $g'$ and thus, it is norm attaining by the assumption. Hence, $\NA^{(2)}(X,\ell_2^2)\neq \emptyset$ by Theorem~\ref{0}.
\end{proof}

Our final goal in the section is to discuss the density of norm attaining operators whose range is a two-dimensional Hilbert space in terms of mates.

\begin{proposition}\label{characterizationdensity}
  Let $X$ be a Banach space. Then the following are equivalent:
  \begin{enumerate}
  \item[(i)]
    $\NA(X,\ell_2^2)$ is dense in $\mathcal{L}(X,\ell_2^2)$.
  \item[(ii)]
    For every $f\in S_{X^*}$ and $g\in B_{X^*}\setminus\{0\}$ such that $\Vert (f,g)\Vert=1$ there are sequences $\{f_n\}$ in $\NA_1(X)$ and $\{g_n\}$ in  $B_{X^*}\setminus\{0\}$ and a rotation $\pi$ on $\ell_2^2$ such that $\Vert f_n+tg_n\Vert\leq \sqrt{1+t^2}$ for all $t\in \R$ and all $n\in \N$, and $\lim_n (f_n,g_n)= \pi \circ (f,g)$.
      \end{enumerate}
\end{proposition}

\begin{proof}
  (i) $\Rightarrow$ (ii):
  Let $f$ and $g$ be as in (ii), and consider the rank-two operator $T=(f,g)$ with $\|T\|=1$. By (i), there is a sequence of norm attaining operators $T_n'=(f_n', g_n')$ converging to~$T$. The $T_n'$ are also of rank two, at least eventually; and we may assume that $\|T_n'\|=1$ for all $n\in \N$ as well. Pick $x_n\in S_X$  such that $\|T_n'(x_n)\|=1$; i.e.,
  $$
  f_n'(x_n)^2 + g_n'(x_n)^2 = 1.
  $$
  Let us consider a rotation $\pi_{\varphi_n}$ by some angle $\varphi_n\in [-\pi, \pi]$ mapping $T_n'(x_n)=(f_n'(x_n), g_n'(x_n))$ to $(1,0)$. By passing to a  subsequence, we may suppose that $\{\varphi_n\}$ converges to some~$\varphi$, and then writing $T_n=(f_n,g_n):= \pi_\varphi\circ (f_n', g_n')$ for every $n\in \N$, we have that the sequence $\{T_n\}$ converges to $\pi_\varphi\circ (f, g)$. Note that $T_n$ belongs to $\NA(X,\ell_2^2)$ for every $n\in \N$ since every $T_n'$ does and, therefore, we have the desired inequality by Proposition~\ref{estimation2} and Theorem~\ref{0}.

  (ii) $\Rightarrow$ (i): By the Bishop-Phelps theorem, one can approximate rank~$1$ operators by norm attaining ones; and by the rotation argument in the proof of Proposition~\ref{prop2.8}, it is enough to show that operators $T=(f,g)$ in $S_{\mathcal{L}(X,\ell_2^2)}$ with  $f\in S_{X^*}$ and $g\in B_{X^*}\setminus\{0\}$ can be approximated. From (ii), there  are sequences $\{f_n\}$ in $\NA_1(X)$ and $\{g_n\}$ in  $B_{X^*}\setminus\{0\}$ such that $\Vert f_n+tg_n\Vert \leq \sqrt{1+ t^2}$ for all $t\in \R$ and $\lim_n (f_n,g_n)=\pi\circ T$, for some rotation~$\pi$.
  By Theorem~\ref{0}, $T_n=(f_n,g_n)$ is norm attaining, hence also $\pi^{-1}\circ T_n\in \NA(X,\ell_2^2)$ and $\pi^{-1}\circ T_n \longrightarrow T$.
\end{proof}

We remark that there are sufficient conditions on a Banach space $X$ expounded in Section~\ref{section:density} to assure that each finite rank operator from $X$ can be approximated by norm attaining finite rank operators; in particular, this is true for $X$ a $C_0(L)$ space, an $L_1(\mu)$ space, a predual of $\ell_1$, or a proximinal subspace of $c_0$ or of $\mathcal{K}(\ell_2)$ of finite codimension. Therefore, for these domain spaces $X$, item~(ii) of Proposition~\ref{characterizationdensity} holds.

\section{A question about Lomonosov's example} \label{sect:Lomonosov}

When one speaks about norm attaining functionals, there is no big difference between the real and the complex case, because a complex functional on a complex space  attains its norm if, and only if, the real part of the functional does and, besides, a complex functional on a complex Banach space is completely determined by its real part. Therefore, if $X$ is a complex space, then the set of complex-linear functionals on $X$ which attain their maximum modulus coincides with the set of those complex-linear functionals on $X$ whose real parts attain their maximum, so this set is dense by the Bishop-Phelps theorem (compare with the situation which occurs when we consider real-linear operators from $X$ to $\C\equiv \ell_2^2$, see Section~\ref{sect:naoperators}).

But in the same papers \cite{BP1, BP2} that deal with norm attaining functionals, Bishop and Phelps considered an analogous question about functionals that attain their maximum on a given set $C$. It is proved that, for a closed bounded convex subset $C$ of a real Banach space, the set of maximum attaining functionals is dense in $X^*$ (in  \cite{BP1} this was just a remark at the end of the paper, saying that the proof may be done in the same way as for norm attaining functionals, and in \cite{BP2} the result is given with all details).

Passing to complex functionals, one cannot speak about the maximal value on a subset $C$, but it is natural to ask if  $\sup_{x \in C} |f(x)|$ is actually a maximum. Let us fix some terminology. For a given subset $C \neq \emptyset$ of a complex Banach space $X$, a non-zero complex functional $f \in X^*$  is said to be a \emph{modulus support functional} for $C$ if there is a point $y \in C$ (called the corresponding \emph{modulus support point} of $C$) such that $|f(y)| = \sup_{x \in C} |f(x)|$. The natural question \cite{Phelps1991} whether for every closed bounded convex subset of a complex Banach space the corresponding set of support functionals is dense in $X^*$ remained open until 2000, when Victor Lomonosov \cite{Lom-2000, Lom-2000+} constructed his striking example of a closed bounded convex subset of the predual space of $H^\infty$  which does not admit any modulus support functionals. A similar construction can be made  \cite{Lom-2001} in the predual $A_*$ of every dual algebra $A$ of operators on a Hilbert space which is not self-adjoint (i.e., there is an operator $T \in A$ such that $T^* \not\in A$), contains the identity operator and such that the spectral radius of every operator
in $A$ coincides with its norm.

By a weak compactness argument, examples of such kind cannot live in a reflexive space. Moreover, they do not exist in spaces with the Radon-Nikod\'{y}m property by Bourgain's result \cite{Bou}, see the argument at the end of page~340 of \cite{Phelps1991}. Therefore, in most classical spaces like $L_p[0, 1]$ with $1<p< \infty$ or $\ell_p$ with $1\le p < \infty$ the complex version of the Bishop-Phelps theorem for subsets is valid. The spaces $\ell_\infty$, $L_\infty[0, 1]$ and $C[0, 1]$ have subspaces isometric to any given separable space, which makes it possible to transfer  Lomonosov's example to these spaces. For the remaining two classical spaces, $c_0$ and $L_1[0, 1]$, the validity of the complex version of the Bishop-Phelps theorem for subsets is an open question.

In the case of the complex space $c_0$, we have an easy way to define a concrete closed, bounded and convex subset $S$ for which we don't know whether its set of modulus support functionals is dense, and not even whether it is non-empty. The first author discussed this example with several colleagues, in particular with Victor Lomonosov, but to no avail. So we decided to use this occasion to appeal to a wider circle of people interested in the subject by publishing the example here.

Let $\D = \{z \in \C \mycolon |z| < 1\}$ be the open unit disk $e_n \in c_0$ be the elements of the canonical basis, and $e_n^* \in \ell_1$ be the corresponding coordinate functionals. For every $z \in  \D$, consider $\varphi_z = \sum_{n=1}^\infty z^n e_n \in c_0$. The set $S \subset c_0$ in question is
\begin{equation}\label{eq:set-S}
S = \cco \{\varphi_z \mycolon z \in  \D\}.
\end{equation}
Remark that, identifying each element $a = (a_1, a_2, \ldots) \in \ell_1$ with the function $f_a$ on the unit disk by the rule $f_a(\zeta) = \sum_{n=1}^\infty a_n \zeta^n$ for all $\zeta\in \D$, we identify $c_0^* = \ell_1$ with the corresponding algebra $\tilde \ell_1$ of analytic functions vanishing at zero and having an absolutely convergent series of Taylor coefficients, equipped with the norm $\|f_a\| = \|a\|_1 =  \sum_{n=1}^\infty |a_n|$.
Taking into account that, in the duality of $c_0$ and~$\ell_1$,
$$
a(\varphi_z ) =  \sum_{n=1}^\infty a_n z^n = f_a(z),
$$
we may identify each element $\varphi_z$ with the evaluation functional $\delta_z$ at the point~$z$ on~$\tilde\ell_1$. Having a look at the papers \cite{Lom-2000, Lom-2000+}, one can see that our $S$ is basically the same as in Lomonosov's example, with the difference that the algebra  $H^\infty$ is substituted by~$\tilde \ell_1$. For every  $a \in \ell_1$, one has that
$$
 \sup_{x \in S} |a(x)| = \sup_{z \in \D} |f_a(z)|,
$$
which is the spectral radius of the element $f_a \in \tilde \ell_1$. Lomonosov uses in his example that in $H^\infty$ the spectral radius of every element is equal to its norm. In $\tilde \ell_1$ this is not the case, which does not permit us to use Lomonosov's argument in our case. Nevertheless, many features survive, which makes the existence of modulus support functionals  very questionable.

At first we remark that by the maximum modulus principle for analytic functions,
$ \sup_{z \in \D} |f_a(z)|$ cannot be attained, so none of the points $\varphi_z$ is a modulus support point. Digging deeper, assume that $y = (y_1, y_2, \ldots) \in S$ is a modulus support point that corresponds to the modulus support functional $b = (b_1, b_2, \ldots) \in \ell_1 \setminus \{0\}$, that is
$$
|b(y)| = \sup_{x \in S} |b(x)| = \sup_{z \in \D} |f_b(z)|.
$$
Pick elements $w_n \in \co\{\varphi_z \mycolon z \in  \D\}$ that converge to $y$,
$w_n = \sum_{k \in \N} w_{n, k} \varphi_{z_k}$, $z_k \in \D$, where $w_{n, k} \ge 0$, $\sum_{k \in \N} w_{n, k} = 1$, and for every $n \in \N$ there is an $m(n)$ such that $w_{n, k} = 0$ for all $k > m(n)$.

Consider the corresponding probability measures $\mu_n = \sum_{k \in \N} w_{n, k} \delta_{z_k} \in C(\bar \D)^*$. By the separability of $C(\bar \D)$, passing to a subsequence, we may assume without loss of generality that the sequence $\{\mu_n\}$ converges in the weak-$*$ topology of $C(\bar \D)^*$ to a Borel probability measure $\mu$ on $\bar \D$. This $\mu$ is related to $y$ as follows: for every $j \in \N$, one has that
\begin{equation*} 
\int_{\bar \D} z^j\, d \mu(z) = \lim_{n \to \infty} \int_{\bar \D} z^j\, d \mu_n(z) = \lim_{n \to \infty} e_j^*(w_n) = e_j^*(y) = y_j,
\end{equation*}
so
\begin{equation*} 
\int_{\bar \D} z^j d \mu(z) \xrightarrow[j \to \infty]{} 0.
\end{equation*}
By a similar argument,
$$
\int_{\bar \D} f_b(z) \,d \mu(z)  = \lim_{n \to \infty} \int_{\bar \D} f_b(z) \,d \mu_n(z) = \lim_{n \to \infty} b(w_n) = b(y),
$$
and consequently
$$
\left|\int_{\bar \D} f_b(z) \,d \mu(z)\right| = |b(y)| = \sup_{z \in {\bar \D}} |f_b(z)|.
$$
Denoting $r = \sup_{z \in {\bar \D}} |f_b(z)|$, we deduce from the above property that  $$\supp \mu \subset \{v \in \bar \D \mycolon |f_b(v)| =  r\} \subset \bar \D \setminus \D$$ and, moreover, the function $f_b$ must take a constant value $\alpha$ on $\supp \mu$ with $|\alpha| = r$. These conditions on $\mu$ and $b$ are very restrictive, and we don't know whether such a wild pair of animals  exists.

We finish the section by emphasizing the question we have been discussing here.

\begin{problem}
Let $S$ be the subset of the complex space $c_0$ given in \eqref{eq:set-S}.
Are the modulus support functionals for $S$ dense in $c_0^*$?
\end{problem}

\bigskip

\noindent \textbf{Acknowledgment.} The authors would like to thank Gilles Godefroy, Rafael Pay\'{a}, and Beata Randrianantoanina for kindly answering several inquiries related to the topics of this manuscript and for providing valuable references.

\end{document}